\documentclass[12pt]{article}
\usepackage{amsmath,graphicx,amssymb,amsthm}
\usepackage{pgf}

\allowdisplaybreaks

\def\amslatex{$\mathcal{A}\kern-.1667em\lower.5ex\hbox{$\mathcal{M}$}\kern-.125em\mathcal{S}$-\LaTeX}

\newtheorem{set}{set}[section]
\newtheorem{Corollary}[set]{Corollary}
\newtheorem{Definition}[set]{Definition}

\newtheorem{Lemma}[set]{Lemma}

\newtheorem{Proposition}[set]{Proposition}
\newtheorem{Remark}[set]{Remark}
\newtheorem{Theorem}[set]{Theorem}
\newcommand{\CA}{$C^*$-algebra}
\newcommand{\ep}{\epsilon}
\newcommand{\dt}{\delta}
\newcommand{\beq}{\begin{eqnarray}}
\newcommand{\eneq}{\end{eqnarray}}
\newcommand{\andeqn}{\,\,\,{\rm and}\,\,\,}
\newcommand{\define}{\mathrel{\hbox{$\equiv$\hskip -.90em \lower .47ex \hbox{$\leftharpoondown$}}}}
\newcommand{\enifed}{\mathrel{\hbox{$\equiv$\hskip -.90em \lower .47ex \hbox{$\rightharpoondown$}}}}

\numberwithin{equation}{section}
\begin{document}
\title{On  generalized universal irrational rotation algebras and the operator $u+v$}
\author{Junsheng Fang\and Chunlan Jiang \and Huaxin Lin \and Feng Xu}
\date{}
 \maketitle

 \begin{abstract}

We introduce a class of generalized universal irrational rotation
 $C^*$-algebras $A_{\theta,\gamma}=C^*(x,w)$ which is
 characterized by the relations $w^*w=ww^*=1$, $x^*x=\gamma(w)$,
 $xx^*=\gamma(e^{-2\pi i\theta}w)$, and $xw=e^{-2\pi i\theta}wx$, where $\theta$ is an irrational number and $\gamma(z)\in C(\mathbb{T})$ is a
 positive function. We characterize  tracial linear functionals, simplicity,
 and $K$-groups of $A_{\theta,\gamma}$ in
 terms of zero points of $\gamma(z)$. We show that if $A_{\theta,\gamma}$ is simple then $A_{\theta,\gamma}$ is an $A{\mathbb T}$-algebra of real rank zero. We  classify $A_{\theta,\gamma}$ in terms of $\theta$ and zero points of $\gamma(z)$.
Let $A_\theta=C^*(u,v)$ be the universal irrational rotation $C^*$-algebra with $vu=e^{2\pi i\theta}uv$.
 Then $C^*(u+v)\cong A_{\theta,|1+z|^2}$.
As an application, we show that $C^*(u+v)$ is a proper simple $C^*$-subalgebra of
$A_\theta$ which has a unique trace,
$K_1(C^*(u+v))\cong \mathbb{Z}$, and there is an order isomorphism
of $K_0(C^*(u+v))$ onto $\mathbb{Z}+\mathbb{Z}\theta$. {Moreover,
$C^*(u+v)$ is a unital simple $A{\mathbb T}$-algebra of real rank zero.}
We also calculate the spectrum and the Brown measure of $u+v$.
 \end{abstract}

\section{Introduction}

The irrational rotation \CA\, $A_{\theta}$ has been one of most studied \CA s. It is known now that $A_\theta$ is a unital simple \CA\, with a unique tracial state. There is an order isomorphism
of $K_0(A_\theta)$ onto $\mathbb{Z}+\mathbb{Z}\theta$ and $K_1(A_\theta)\cong\mathbb{Z}^2$ (\cite{PV, Ri}). Moreover,
$A_\theta$ is a unital simple $A{\mathbb T}$-algebra of real rank zero ~\cite{EE}.

Let $u$ and $v$ be two unitary generators of the universal irrational rotation $C^*$-algebra  $A_\theta$  such that $vu=e^{2\pi i\theta}{uv}$. Then $u+v$ is {an ab}normal operator of $A_\theta$ and $C^*(u+v)$ is a proper $C^*$-subalgebra of $A_\theta$. In this paper, we study the algebraic structure of $C^*(u+v)$ and the spectral theory of $u+v$. Our motivation comes from our attempt to relate the theory of strongly irreducible operators relative to
${\rm II}_1$ factors  with irreducible subfactors (cf. Prop. \ref{indexn} and the question that follows).

In fact, we study a class of generalized universal irrational rotation  $C^*$-algebras $A_{\theta,\gamma}=C^*(x,w)$, which is
the universal $C^*$-algebra satisfying the
following properties:
\begin{equation}\label{E:11}
w^*w=ww^*=1,
\end{equation}
\begin{equation}\label{E:12}
x^*x=\gamma(w),
\end{equation}
\begin{equation}\label{E:13}
xx^*=\gamma(e^{-2\pi i\theta}w),
\end{equation}
\begin{equation}\label{E:14}
xw=e^{-2\pi i\theta}wx,
\end{equation}
where $\theta\in (0,1)$ and $\gamma(z)\in C(\mathbb{T})$ is a positive continuous function of the unit circle ${\mathbb T}$. As we will see that $C^*(u+v)\cong
 A_{\theta,|1+z|^2}$. If $\theta$ is an irrational number and $\gamma(z)\equiv 1$, then $A_{\theta,\gamma}$ is the  irrational rotation $C^*$-algebra $A_\theta$.  In fact, if $\gamma$ is invertible, then $A_{\theta, \gamma}=A_\theta.$  However, the main interest of this paper is to study $A_{\theta,\gamma}$ when  the set of the zero points of $\gamma(z)$ is nonempty.

It turns out that, when $\theta$ is fixed,  the $C^*$-algebra $A_{\theta,\gamma}$ only depends on the set of zero points and therefore the algebraic property of $A_{\theta,\gamma}$ is completely determined by the zero points of $\gamma(z)$. For example, we characterize simplicity and uniqueness of trace of  $A_{\theta,\gamma}$ as follows. Let $Y$ be the set of zero points of $\gamma(z)$ and let $\phi:\mathbb{T}\rightarrow\mathbb{T}$ be the rotation of the unit circle determined by $\theta$, i.e., $\phi(z)=e^{2\pi i\theta}z$. Denote by $Orb(\xi)=\{\phi^n(\xi):\,n\in\mathbb{Z}\}$ for $\xi\in \mathbb{T}$. Then the following properties are equivalent:
\begin{enumerate}
\item ${A}_{\theta, \gamma}$ is simple;
\item ${A}_{\theta, \gamma}$ has a unique tracial state;
\item $\phi^n(Y)\cap Y=\emptyset$ for all integer $n\not=0;$
\item For each $\xi\in {\mathbb T},$ $Orb(\xi)\cap Y$ contains at most one point.
\end{enumerate}
If $Y$ is not empty, then $K_1({A}_{\theta, \gamma})\cong\mathbb{Z}$ and $K_0({A}_{\theta, \gamma})$ is determined by the following splitting exact sequence
\[
0\to {\mathbb Z}\to K_0({A}_{\theta, \gamma})\to C(Y, {\mathbb Z})\to 0.
\]
We also show that if ${A}_{\theta, \gamma}$ is simple, then ${A}_{\theta, \gamma}$ has tracial rank zero and is an inductive limit of recursive subhomogenuous \CA s. As a result, the classification of  ${A}_{\theta,\gamma}$ falls into Elliott's classification program. Indeed, we obtain the following result. Let $\theta_1$ and $\theta_2$ be two irrational numbers,
$\gamma_1$ and $\gamma_2\in C({\mathbb T})$ be non-negative functions and let $Y_i$ be the set of zeros of $\gamma_i,$ $i=1,2.$
Let $\phi_1,\, \phi_2: {\mathbb T}\to {\mathbb T}$ be rotations of the unit circle determined by
$\theta_1$ and $\theta_2$ respectively.
Suppose that $\phi^n(Y_i)\cap Y_i=\emptyset$ for all integers $n\not=0,$  $i=1,2.$
Then  ${A}_{\theta_1, \gamma_1}\cong {A}_{\theta_2, \gamma_2}$  if and only if
the following hold:
\[
\theta_1=\pm \theta_2 mod(\mathbb Z)\andeqn C(Y_1, {\mathbb Z})/{\mathbb Z}\cong C(Y_2, {\mathbb Z})/{\mathbb Z}.
\]
In particular, when  $\gamma_1$ has only finitely many zeros, then
${A}_{\theta_1, \gamma_1}\cong {A}_{\theta_2, \gamma_2}$  if and only if $\theta_1=\pm \theta_2 mod(\mathbb Z)$ and
$\gamma_2$ has the same number of zeros as those of $\gamma_1.$

A special case of interest is
\[C^*(u+v)=C^*(u+v,u^*v)=C^*(u(1+u^*v),u^*v)\cong A_{\theta,\gamma},\]
where $\gamma(z)=|1+z|^2$. As an application of the above results of generalized universal irrational rotation $C^*$-algebras, we show that $C^*(u+v)$ is a proper simple $C^*$-subalgebra of
$A_\theta$ which has a unique trace,
$K_1(C^*(u+v))\cong\mathbb{Z}$, and there is an order isomorphism
of $K_0(C^*(u+v))$ onto $\mathbb{Z}+\mathbb{Z}\theta$. Moreover,
$C^*(u+v)$ is a unital simple $A{\mathbb T}$-algebra with real rank zero. Therefore, $C^*(u+v)$ has tracial rank zero.

The second part of the paper is to study the spectrum of $u+v$, which is motivated by the ``the Ten Martini
Problem" on the almost Mathieu operator.
In mathematical physics, the almost Mathieu operator is given by

\[(H_{\lambda,\theta,\beta} u)(n) = u(n+1) + u(n-1) + 2 \lambda
\cos(2\pi (n\theta +\beta)) u(n), \] acting as a self-adjoint
operator on the Hilbert space $\ell^2(\mathbb{Z})$. Here
$\theta,\beta, \lambda \in\mathbb{R}$ are parameters.  Almost
Mathieu operator was firstly introduced by R. Peierls \cite{Pe} and has
been extensively studied (see \cite{La} for a recent historical account and for the physics background). In pure mathematics, its importance
comes from the fact of being one of the best-understood examples
of an ergodic Schr\"{o}dinger operator.
 For example, three
problems (now all solved) of Barry Simon's fifteen
problems~\cite{Si} about Schr\"{o}dinger operators``for the
twenty-first century" featured the almost Mathieu operator. The
fourth problem in~\cite{Si} (known as the ``the Ten Martini
Problem" after Kac and Simon)  conjectures that the spectrum of
the almost Mathieu operator is a Cantor set for all $\lambda\neq
0$ and irrational numbers $\theta$. Recently, Avila and
Jitomirskaya 
confirmed this conjecture in~\cite{AJ}.
For a history of this problem and earlier partial results see
\cite{La,BS,Si,HS1,CEY,AK,Pu}.

 Recall
that the  irrational rotation $C^*$-algebra  $A_\theta$ can be represented on $\ell^2(\mathbb{Z})$, by
mapping $u$ into the bilateral shift (taking $\phi$ into
$(\phi(n-1))_{n\in\mathbb{Z}}$), and $v$ into the operation of
multiplication by $e^{2\pi in\theta}$ (taking $\phi$ into $e^{2\pi
in\theta}(\phi(n))_{n\in\mathbb{Z}}$), and then the polynomial
$(u+\lambda e^{2\pi i\beta} v)+(u+\lambda e^{2\pi i\beta} v)^*$ is mapped into the bounded self-adjoint operator
$H_{\lambda,\theta,\beta}$. Since $A_\theta$ is simple (when $\theta$ is
irrational), the spectrum of $H_{\lambda,\theta,\beta}$ is the same as the
spectrum of the element $(u+\lambda e^{2\pi i\beta} v)+(u+\lambda e^{2\pi i\beta} v)^*$.
A natural question is that what is the spectrum of
$u+\lambda e^{2\pi i\beta}v$? If $\theta$ is an irrational number, then by the uniqueness of $A_\theta$ the spectrum of $u+\lambda e^{2\pi i\beta}v$ is the same as $u+|\lambda|v$. So from now on, we always assume that $\lambda>0$ and $\beta=0$.

Let  $\tau$  be the unique tracial state on
$A_\theta$. By the GNS-construction, we obtain a
faithful representation $\pi$ of $A_\theta$ on
$L^2(A_\theta,\tau)$. The weak operator closure of
$\pi(A_\theta)$ is the hyperfinite ${\rm II}_1$ factor
$R$. Since the spectrum of $u+\lambda v$ is same as the spectrum
of $\pi(u+\lambda v)$ in $R$, we need only to calculate the
spectrum of $\pi(u+\lambda v)$ in $R$. In the following we
identify $A_\theta$ with $\pi(A_\theta)$ and
thus identify $u+\lambda v$ with $\pi(u+\lambda v)$.

One of the main results of the present paper is that the spectrum
of $u+\lambda v$ is given by
\[
\sigma(u+\lambda v)=
\begin{cases}
\mathbb{T}& 0<\lambda<1,\\
\overline{B(0,1)}&\lambda=1,\\
\lambda \mathbb{T}& \lambda>1.
\end{cases}
\]
Another result of spectral theory is related to {the} Brown measure.
L. G. Brown introduced in the paper~\cite{Br} a spectral
distribution measure $\mu_T$ for not necessarily  normal operators
$T$ in a von Neumann algebra $M$ with a faithful normal tracial
state $\tau$, which is called the Brown measure of $T$. Recently, U. Haagerup and H. Schultz~\cite{HS2} proved
a remarkable result about the existence of nontrivial
hyperinvariant subspaces of operators in type ${\rm II}_1$ factors.
They proved that if the support of $\mu_T$ contains more that two
points, then $T$ has a nontrivial hyperinvariant space. However,
the calculation of Brown measures of nonnormal operators is
difficult in general (see\cite{HL,BL,FHM}). In particular,
Haagerup and Larsen in~\cite{HL} showed that the Brown measure of
the sum of two free Haar unitary operator $T=u_1+u_2$ is rotation
invariant, has support equal to $\overline{B(0,\sqrt{2})}$
($=\sigma(T)$), and has radial density
\[
f_T(r)=
\begin{cases}
\frac{4}{4\pi(4-r^2)^2},& 0<r<\sqrt{2}\\
0,&\text{otherwise.}
\end{cases}
\]
In section 12, we will show that the Brown measure of $u+v$
(in $R$) is  the Haar measure on the unit circle.

This paper is organized as follows. In section 2 we introduce the
class of generalized universal irrational rotation
 $C^*$-algebras $A_{\theta,\gamma}=C^*(x,w)$.
 We prove that, in fact, $A_{\theta, \gamma}$ can be viewed as a $C^*$-subalgebra of $A_\theta.$
 We also fix some notation that will be used in the later sections.
 In section 3, we
 give some descriptions of the tracial state space of $A_{\theta, \gamma}$
  in terms of zero points of $\gamma(z)$. In particular, we show that
 $A$ has a unique tracial state if and only if each rotation orbit
 contains at most one zero point of $\gamma.$
  In section 4, we characterize simplicity of  $A_{\theta,\gamma}$ in terms of zero points of $\gamma(z)$. We show that $A_{\theta, \gamma}$ is simple if and only if it has a unique tracial state which is also equivalent to the condition that each rotation orbit contains at most one zero point of $\gamma.$
 In section 5, we
construct Rieffel's projections in every simple generalized universal
irrational rotation algebra $A_{\theta,\gamma}$.
In section 6, we calculate $K$-groups of
$A_{\theta,\gamma}$.
In section 7, using results of section 3-6 and recent development in the Elliott's classification program, we show that when ${A}_{\theta, \gamma}$ is simple, then ${A}_{\theta, \gamma}$ is an $A{\mathbb T}$-algebra of real rank zero.
We obtain a classification result of simple $C^*$-algebras of $A_{\theta,\gamma}$ in terms $\theta$ and zero points of $\gamma(z)$.
In section 8 we prove that the von Neumann
subalgebra generated by $u+\lambda v$ is $R$ for all
$0<\lambda<\infty$, and the $C^*$-algebra generated by $u+\lambda
v$ is $C^*(u,v)$ if $\lambda\neq 1$. However, for $\lambda=1$,
$C^*(u+v)$  is isomorphic to $A_{\theta, |1+z|^2}.$ Therefore
$C^*(u+v)$ is a unital simple $A{\mathbb T}$-algebra of real rank zero which has
$K_1(C^*(u+v))\cong {\mathbb Z}$ and $K_0(u+v)$ is order isomorphic to ${\mathbb Z}+{\mathbb Z}\theta.$
 In particular, $C^*(u+v)$ is not $\ast$-isomorphic to $C^*(u,v)$.

In section 9 we show that the spectral radius of $u+\lambda v$ is
$1$ if $0<\lambda\leq 1$.  A key idea in the calculation  is using
Birkhoff's Ergodic theorem. Then in section 10 we show that the
relative commutant of $u+v$ in $R$ does not contain any nontrivial
idempotent. By the Riesz spectral decomposition theorem, the
spectrum of $u+v$ is connected.
  Combining the fact that the spectrum of $u+v$ is
rotation symmetric, in section 11 we obtain that
$\sigma(u+v)=\overline{B(0,1)}$.
 We show that the spectral radius of $(u+\lambda v)^{-1}$ is less or equal
  than 1 for $0<\lambda<1$, which implies that $\sigma(u+\lambda v)$ is contained in
   the unit circle $\mathbb{T}$. Since the spectrum of $u+\lambda v$ is rotation symmetric,
 $\sigma(u+\lambda v)=\mathbb{T}$. By the symmetry of $u$ and $v$,
 $\sigma(u+\lambda v)=\lambda \sigma(\lambda^{-1}u+v)=\lambda \mathbb{T}$ for $\lambda>1$.
In section 12, we calculate Brown measure of $u+\lambda v$.

{\bf Acknowledgements:} The authors thank Professor Guihua Gong for many valuable discussions on this paper.

\section{Generalized universal irrational rotation $C^*$-algebras}
Let $u$ and $v$ be two unitary generators of the universal irrational rotation $C^*$-algebra $A_\theta$ such that $vu=e^{2\pi i\theta}uv$.
To study the properties of $C^*$-algebras generated by $u+
v$, we will consider the universal $C^*$-algebra satisfying the
following properties:
\begin{equation}\label{U1}
w^*w=ww^*=1,
\end{equation}
\begin{equation}\label{U2}
x^*x=\gamma(w),
\end{equation}
\begin{equation}\label{U3}
xx^*=\gamma(e^{-2\pi i\theta}w),
\end{equation}
\begin{equation}\label{U4}
xw=e^{-2\pi i\theta}wx,
\end{equation}
where $\gamma(z)\in C(\mathbb{T})$ is a positive function.

 A $C^*$-algebra $A_{\theta,\gamma}$ is universal for the
above relations provided that it is generated by operators $x,w$
satisfying (\ref{U1})-(\ref{U4}) and whenever
$\mathfrak{A}=C^*(x',w')$ is another $C^*$-algebra satisfying
(\ref{U1})-(\ref{U4}), there is a homomorphism of
$A_{\theta,\gamma}$ onto $\mathfrak{A}$ which carries
$x$ to $x'$ and $w$ to $w'$.  By (\ref{U1}), $w$ is a unitary
operator. So (\ref{U2}) implies that $\|x\|\leq \|\gamma\|^{1/2}$.
We may consider the collection of all
 operators $x_\alpha, w_\alpha$ in $B(H_{{\alpha}})$ satisfying
(\ref{U1})-(\ref{U4}). Then form the operator
\[
x=\sum\oplus x_\alpha\quad\text{and}\quad w=\sum\oplus w_\alpha.
\]
Let $A_{\theta,\gamma}=C^*(x,w)$. Then
$A_{\theta,\gamma}$ is the desired universal algebra.  Note that if
$\gamma(z)\equiv 1$, then $A_{\theta,\gamma}$ is
precisely the universal irrational rotation algebra. So we call
$A_{\theta,\gamma}$ a \emph{generalized universal
irrational rotation algebra}.

Let $A_\theta$ be the universal irrational rotation $C^*$-algebra with two
unitary generators $u,v$ with
$vu=e^{2\pi i\theta}uv$.
Then $u\gamma(v)^{1/2}$ and $v$ satisfy
(\ref{U1})-(\ref{U4}).
{So there is a $\ast$-homomorphism from $A_{\theta, \gamma}$ onto the $C^*$-subalgebra of $A_\theta$ generated by $u\gamma(v)^{1/2}$ and $v$. We will show that we may view $A_{\theta, \gamma}$ as  the $C^*$-subalgebra of $A_\theta$ generated by $u\gamma(v)^{1/2}$ and $v$ and $C^*(u+v)\cong A_{\theta, |1+z|^2}$.}

By (\ref{U1})-(\ref{U4}) and simple calculations, we have the
following equations.

\begin{equation}\label{U5}
x^*w=e^{2\pi i\theta}wx^*,
\end{equation}

\begin{equation}\label{U4-f(z)}
xf(w)=f(e^{-2\pi i\theta}w)x,\quad\forall f(z)\in C(\mathbb{T}),
\end{equation}

\begin{equation}\label{U5-f(z)}
x^*f(w)=f(e^{2\pi i\theta}w)x^*,\quad\forall f(z)\in
C(\mathbb{T}),
\end{equation}

\begin{equation}\label{E:x*rxr}
(x^*)^rx^r=\gamma\left(e^{2\pi
i(r-1)\theta}w\right)\gamma\left(e^{2\pi
i(r-2)\theta}w\right)\cdots \gamma(w),
\end{equation}

\begin{equation}\label{E:xrx*r}
x^r(x^*)^r=\gamma\left(e^{-2\pi
ir\theta}w\right)\gamma\left(e^{-2\pi i(r-1)\theta}w\right)\cdots
\gamma\left(e^{-2\pi i\theta}w\right).
\end{equation}

We apply the universal property to obtain certain special
automorphisms of $A_{\theta,\gamma}$. For any constant
$\lambda=e^{2\pi it}$ on the unit circle, the pair $(\lambda x,w)$
also satisfy  (\ref{U1})-(\ref{U4}). Thus there is an endomorphism
of $A_{\theta,\gamma}$ such that $\rho_t(x)=\lambda x$
and $\rho_t(w)=w$. By symmetry, $\rho_{-t}(x)=\bar{\lambda} x$ and
$\rho_{-t}(w)=w$. Hence,
$\rho_{-t}(\rho_t(x))=\rho_t(\rho_{-t}(x))=x$ and
$\rho_{-t}(\rho_t(w))=\rho_t(\rho_{-t}(w))=w$. This implies that
$\rho_t$ is an automorphism of $A_{\theta,\gamma}$.

For each fixed $a$ in $A_{\theta,\gamma}$, the map from
$[0,1]$ to $A_{\theta,\gamma}$ given by
$f(t)=\rho_t(a)$ is norm continuous. To verify this, notice
that it is true for all noncommutative polynomials in
$x,x^*,w,w^*$. These are dense and automorphisms are contractive;
so the rest follows from a simple approximation argument.

Define a map of $A_{\theta,\gamma}$ into itself by
\[
\Phi(a)=\int_0^1\rho_t(a)dt.
\]
Then the integral makes sense as Riemann sum because the integrand
is a norm continuous function. By (\ref{U1})-(\ref{E:xrx*r}) and simple
calculations, we can see that the following set
\[
\left\{\sum_{n=1}^Nx^nf_n(w)+f_0(w)+\sum_{n=1}^Nf_{-n}(w)(x^*)^n|\,N\in
\mathbb{N},\, f_n(z),f_{-n}(z)\in C(\mathbb{T})\right\}
\]
is dense in $A_{\theta,\gamma}$.

The proof of the following proposition is similar to the proof of Theorem VI.1.1 of~\cite{Da}. For the
sake of completeness, we include a detailed proof.
\begin{Proposition}\label{P:U Phi}
The map $\Phi$ is a faithful conditional expectation of
$A_{\theta,\gamma}$ onto $C^*(w)$ such that
$\Phi(x^kf(w))=\Phi(f(w)(x^*)^k)=0$ for all $f(z)\in
C(\mathbb{T})$ and $k\in \mathbb{N}$.
 In addition, for every $a\in A_{\theta,\gamma}$,
\[
\Phi(a)=\lim_{n\rightarrow\infty}\frac{1}{2n+1}\sum_{j=-n}^nw^ja(w^*)^j.
\]
\end{Proposition}
\begin{proof}
Since, $\|\rho_t(a)\|=\|a\|,$
$$
\left\|\sum_{j=1}^n \rho_{t_j}(a) \beta_j\right\|\le \|a\|
$$
for any scalar $0\le \beta_j\le 1$ such that
$\sum_{j=1}^n \beta_j=1.$ It follows that
$$
\|\Phi(a)\|=\left\|\int_0^1\rho_t(a)dt\right\|\le \|a\|.
$$
We conclude that
$\|\Phi\|\leq 1$. Since $\Phi(1)=1$, $\|\Phi\|=1$. Since
$\rho_t(w)=w$ for all $t$, $\rho_t(a)=a$ for all $a\in C^*(w)$.
Hence $\Phi(a)=a$ for all $a\in C^*(w)$. By the definition of
$\Phi$,
 \[\Phi(a_1aa_2)=\int_0^1 \rho_t(a_1aa_2)dt=\int_0^1 \rho_t(a_1)\rho_t(a)\rho_t(a_2)dt=\int_0^1 a_1\rho_t(a)a_2dt=a_1\Phi(a)a_2\]
for all $a_1,a_2\in C^*(w)$ and $a\in
A_{\theta,\gamma}$.

Suppose $a=x^kf(w)$ for $f(z)\in C(\mathbb{T})$ and
$k\in\mathbb{N}$. Then
\[\Phi(a)=\int_0^1\rho_t(x^kf(w))dt=\int_0^1\rho_t(x^k)\rho_t(f(w))dt=\int_0^1e^{2\pi ikt}x^kf(w)dt=\left(\int_0^1e^{2\pi ikt}dt\right)a=0.\]
 Suppose $a=f(w)(x^*)^k$
for $f(z)\in C(\mathbb{T})$ and $k\in\mathbb{N}$. Then
\[\Phi(a)=\int_0^1\rho_t(f(w)(x^*)^k)dt=\int_0^1\rho_t(f(w))\rho_t((x^*)^k)dt=\int_0^1e^{-2\pi ikt}f(w)(x^*)^kdt=\left(\int_0^1e^{-2\pi ikt}dt\right)a=0.\]
  Since $\|\Phi\|=1$,
$\Phi(A_{\theta,\gamma})\subseteq C^*(w)$. By Tomiyama's
Theorem~\cite{To}, $\Phi$ is a conditional expectation of
$A_{\theta,\gamma}$ onto $C^*(w)$. If $a$ is positive
and nonzero, then $\rho_t(a)$ is positive and nonzero for all $t$.
Thus the integral $\Phi(a)$ is positive and nonzero. Hence $\Phi$
is faithful.

Suppose $a=x^kf(w)$ for $f(z)\in C(\mathbb{T})$ and
$k\in\mathbb{N}$.  By equations (\ref{U4}) and (\ref{U5}),
\[
\lim_{n\rightarrow\infty}\frac{1}{2n+1}\sum_{j=-n}^nw^ja(w^*)^j=
\lim_{n\rightarrow\infty}\frac{1}{2n+1}\sum_{j=-n}^n e^{2\pi
ijk\theta}a=\lim_{n\rightarrow\infty}\frac{1}{2n+1}\left(\frac{\sin(2n+1)\pi
k\theta}{\sin\pi k\theta}\right)a=0.
\]
Hence
\[
\Phi(a)=\lim_{n\rightarrow\infty}\frac{1}{2n+1}\sum_{j=-n}^nw^ja(w^*)^j=0.
\]
Similarly, we can show that if $a=f(w)(x^*)^k$ for $f(z)\in
C(\mathbb{T})$ and $k\in\mathbb{N}$ then
\[
\Phi(a)=\lim_{n\rightarrow\infty}\frac{1}{2n+1}\sum_{j=-n}^nw^ja(w^*)^j=0.
\]
If $a=f(w)$ for some $f(z)\in C(\mathbb{T})$, then
\[
\Phi(a)=\lim_{n\rightarrow\infty}\frac{1}{2n+1}\sum_{j=-n}^nw^ja(w^*)^j=a.
\]
 By linearity and continuity, this formula is valid for
all $a$ in $A_{\theta,\gamma}$.
\end{proof}

\begin{Corollary}\label{C:rho t}
$\forall a\in A_{\theta,\gamma}$, $\rho_t(a)=a$ for all
$0\leq t\leq 1$ if and only if $a\in C^*(w)$.
\end{Corollary}
\begin{proof}
If $a\in  A_{\theta,\gamma}$ and $\rho_t(a)=a$ for all
$0\leq t\leq 1$, then $a=\Phi(a)\in C^*(w)$ by
Proposition~\ref{P:U Phi}. Conversely, since $\rho_t(w)=w$ for all
$t$, $\rho_t(a)=a$ for all $t$ and $a\in C^*(w)$.
\end{proof}

Let $m=dz/{2\pi}$ be the unique Haar measure on $\mathbb{T}$.
\begin{Remark}
\emph{ If $\gamma(z)\in C(\mathbb{T})$ is a positive function with
$m(\{z|\gamma(z)=0\})=0$, then (\ref{U3}) can be replaced by a
weaker condition
\begin{equation}\label{U3'}
xx^*\in C^*(w).
\end{equation}
To see this, let $xx^*=h(w)$ for some $h(z)\in C(\mathbb{T})$.
Then by (\ref{U5})
\[
\gamma(w)^2=x^*xx^*x=x^*h(w)x=h(e^{2\pi i\theta}w)x^*x=h(e^{2\pi
i\theta}w)\gamma(w).
\]
Hence $\gamma(z)^2=h(e^{2\pi i\theta}z)\gamma(z)$. Let
$E=\{z|\gamma(z)=0\}$. Then for $z\in \mathbb{T}\setminus E$,
$\gamma(z)=h(e^{2\pi i\theta}z)$. Since $m(\mathbb{T}\setminus
E)=1$, $\gamma(z)=h(e^{2\pi i\theta}z)$ for all $z\in \mathbb{T}$.
Thus $h(z)=\gamma(e^{-2\pi i\theta}z)$, which is (\ref{U3}).
 }
\end{Remark}

Note that in the  irrational rotation $C^*$-algebra
$C^*(u,v)$ with
 $vu=e^{2\pi i\theta}uv$, $u\gamma(v)^{1/2}$ and $v$
satisfy (\ref{U1})-(\ref{U4}). So there exists a homomorphism
$\varphi$ from $A_{\theta,\gamma}$ onto
$C^*(u\gamma(v)^{1/2},v)$ such that $\varphi(x)=u\gamma(v)^{1/2}$
and $\varphi(w)=v$. Since the spectrum $\sigma(v)$ is
$\mathbb{T}$, $\sigma(w)=\mathbb{T}$. Hence $C^*(w)\cong
C(\mathbb{T})$. In the following, we identify $C^*(w)$ with
$C(\mathbb{T})$. Let $\rho$ be the state on $C(\mathbb{T})$
induced by the Haar measure $m$ on $\mathbb{T}$. Then $\rho$ is
faithful on $C^*(w)$.

\begin{Lemma}\label{L: U trace}
For $a\in A_{\theta,\gamma}$, let $\tau(a)=\rho\cdot \Phi(a)$.
Then $\tau$ is a faithful trace on $A_{\theta,\gamma}$.
\end{Lemma}
\begin{proof}
Since $\rho$ is a faithful state on $C^*(w)$ and $\Phi$ is a
faithful conditional expectation of $A_{\theta,\gamma}$
onto $C^*(w)$, $\tau$ is a faithful state on
$A_{\theta,\gamma}$. We only need to verify $\tau$ is a
trace. Note that the following set
\[
\left\{\sum_{n=1}^Nx^nf_n(w)+f_0(w)+\sum_{n=1}^Nf_{-n}(w)(x^*)^n|\,N\in
\mathbb{N},\, f_n(z),f_{-n}(z)\in C(\mathbb{T})\right\}
\]
is dense in $A_{\theta,\gamma}$. By boundedness,
linearity and positivity of $\tau$, we need only to verify
$\tau(ab)=\tau(ba)$ for the following two cases.

Case 1. $a=x^rf(w)$, $b=x^sg(w)$, $r,s\geq 0$. If $r+s=0$, i.e.,
$r=s=0$, then $\tau(ab)=\tau(ba)$ is trivial. Suppose $r+s>0$.
Then
\[
\tau(ab)=\tau(x^rf(w)x^sg(w))=\tau(x^{r+s}f(e^{2\pi
is\theta}w)g(w))=\rho(\Phi(x^{r+s})f(e^{2\pi is\theta}w)g(w))=0,
\]
and
\[\tau(ba)=\tau(x^sg(w)x^rf(w))=\tau(x^{r+s}g(e^{2\pi
ir\theta}w)f(w))=\rho(\Phi(x^{r+s})g(e^{2\pi
ir\theta}w)f(w))=0.\] So $\tau(ab)=\tau(ba)$.

Case 2. $a=x^rf(w)$, $b=g(w)(x^*)^s$, $r,s\geq 0$. If $r>s$,
then
\[
\tau(ab)=\tau(x^rf(w)g(w)(x^*)^s)=\tau(x^{r-s}f(e^{-2\pi
is\theta}w)g(e^{-2\pi
is\theta}w)x^s(x^*)^s)\]\[=\rho(\Phi(x^{r-s})f(e^{-2\pi
is\theta}w)g(e^{-2\pi is\theta}w)x^s(x^*)^s)=0,
\]
and
\[
\tau(ba)=\tau(g(w)(x^*)^sx^rf(w))=\tau(g(w)(x^*)^sx^sf(e^{-2\pi
i(r-s)\theta}w)x^{r-s})\]\[=\rho(g(w)(x^*)^sx^sf(e^{-2\pi
i(r-s)\theta}w)\Phi(x^{r-s}))=0.
\]
So $\tau(ab)=\tau(ba)$. Similarly, we can show that if $r<s$ then $\tau(ab)=\tau(ba)$. If $r=s$, then
 we have

\[
\tau(ba)=\rho(ba)=\rho(g(w)(x^*)^rx^rf(w))=\rho(g(w)f(w)\gamma\left(e^{2\pi
i(r-1)\theta}w\right)\gamma\left(e^{2\pi
i(r-2)\theta}w\right)\cdots \gamma(w)),
\]
\[
=\int_{\mathbb{T}}f(z)g(z)\gamma\left(e^{2\pi
i(r-1)\theta}z\right)\gamma\left(e^{2\pi
i(r-2)\theta}z\right)\cdots \gamma(z) dm(z),
\]
\[
\tau(ab)=\rho(ab)=\rho(x^rf(w)g(w)(x^*)^r)=\rho(f(e^{-2\pi
ir\theta}w)g(e^{-2\pi ir\theta}w)x^r(x^*)^r)
\]\[=\rho(f(e^{-2\pi
ir\theta}w)g(e^{-2\pi ir\theta}w)\gamma\left(e^{-2\pi
i\theta}w\right)\gamma\left(e^{-2\pi i2\theta}w\right)\cdots
\gamma\left(e^{-2\pi ir\theta}w\right))
\]
\[
=\int_{\mathbb{T}}f(e^{-2\pi ir\theta}z)g(e^{-2\pi
ir\theta}z)\gamma\left(e^{-2\pi
ir\theta}\cdot e^{2\pi i(r-1)\theta}z \right)\gamma\left(e^{-2\pi
ir\theta}\cdot e^{2\pi i(r-2)\theta}z\right)\cdots \gamma(e^{-2\pi ir\theta}z) dm(z).
\]

Since $m$ is the Haar measure on $\mathbb{T}$,
$\tau(ab)=\tau(ba)$.

\end{proof}

\begin{Theorem}\label{T:universal property}
The homomorphism $\varphi$ from $A_{\theta,\gamma}$ onto
$C^*(u\gamma(v)^{1/2},v)$ such that $\varphi(x)=u\gamma(v)^{1/2}$
and $\varphi(w)=v$ is an isomorphism.
\end{Theorem}
\begin{proof}
Consider the GNS-construction of $A_{\theta,\gamma}$
with respect to the faithful trace $\tau$. Then we may assume that
$A_{\theta,\gamma}$ faithfully acts on the Hilbert space
$L^2(A_{\theta,\gamma},\tau)$. Let $\tau'$ be the unique
trace on $C^*(u,v)$, and let $x'=u\gamma(v)$, $w'=v$. For a
noncommutative polynomial $p$ in four variables, we have
$\tau(p(x,x^*,w,w^*))=\tau'(p(x',(x')^*,w',(w')^*))$. Hence the
operator $U:\, p(x,x^*,w,w^*)\rightarrow p(x',(x')^*,w',(w')^*)$
extends to a unitary operator from
$L^2(A_{\theta,\gamma},\tau)$ onto
$L^2(C^*(u\gamma(v)^{1/2},v),\tau')$. So $\varphi(x)=U^*xU$ is an
isomorphism.
\end{proof}

In what follows, we will identify $A_{\theta, \gamma}$ with
the $C^*$-subalgebra of $A_\theta$ generated by $u\gamma^{1/2}(v)$ and $v.$
We will take advantage of the knowledge of $A_\theta$ to study $A_{\theta, \gamma}.$
We will use the following conventions:

\begin{Definition}
{\rm
We may view $A_{\theta}=C({\mathbb T})\rtimes_\phi {\mathbb Z},$ where
$\phi: {\mathbb T}\to {\mathbb T}$ is defined by
$\phi(z)=e^{2\pi i \theta} z$ for all $z\in {\mathbb T}.$
Define $\alpha_\theta: C({\mathbb T})\to C({\mathbb T})$ by
$\alpha_{\theta}(f)=f\circ \phi$ for all $f\in C({\mathbb T}).$
Denote by $u$ the unitary in $A_{\theta}$ implementing $\alpha_\theta,$ i.e.,
$u^*fu=\alpha_\theta(f)=f\circ \phi$ for all $f\in C({\mathbb T}).$

Let $\gamma: {\mathbb T}\to R_+$ be a nonnegative continuous function and let
$$
Y=\{z\in {\mathbb T}: \gamma(z)=0\}.
$$

Viewing $A_{\theta, \gamma}$ as a $C^*$-subalgebra of $A_\theta,$ it  is easy to check that
$$
A_{\theta, \gamma}=C^*(C({\mathbb T}), uC_0({\mathbb T}\setminus Y)),
$$
the $C^*$-subalgebra of $A_\theta$ generated by $C({\mathbb T})$ and
$\{uf: f\in C_0({\mathbb T}\setminus Y)\}.$

Let $\xi\in {\mathbb T}$ denote by
$$
Orb(\xi)=\{\phi^n(\xi): n\in {\mathbb Z}\}
$$
the orbit of $\xi$ under the rotation $\phi.$
}

\end{Definition}

The following is an easy fact:

\begin{Proposition}\label{NP}
Let $\theta\in (0,1)$ be an irrational number and let $Y\subset {\mathbb T}$ be a subset. Then the following are equivalent:

\begin{enumerate}
\item[(1).] $\phi^n(Y)\cap Y=\emptyset$ for any integer $n\not=0;$

\item[(2).] For each $\xi\in {\mathbb T},$ $Orb(\xi)\cap Y$ contains at most one point;
    \item[(3).] $Y_1\cap Y_2=\emptyset,$
    where $Y_1=\cup_{n\ge 0}\phi^n(Y)$ and
    $Y_2=\cup_{k\ge 1} \phi^{-k}(Y).$

\end{enumerate}

\end{Proposition}

\begin{proof}

(1) $\Rightarrow$ (2):  Suppose that $\phi^{k_1}(\xi),\phi^{k_2}(\xi) \in Y$ for
integers $k_1\not=k_2.$ Then $\phi^{k_1}\xi\in \phi^{k_1-k_2}(Y)\cap Y.$ This is contradiction. So (2) holds.

(2) $\Rightarrow$ (3): If $\xi\in Y_1\cap Y_2,$ then
there are $\xi_1, \xi_2\in Y$ such that
$\xi=e^{2\pi i n\theta}\xi_1=e^{-2\pi i k\theta} \xi_2$ for some
$n\ge 0$ and $k\ge 1.$  It follows that
$\xi_1\in Y$ and $e^{2\pi i (n+k)\theta}\xi_1\in Y.$ By (2), this is impossible. So (3) holds.

(3) $\Rightarrow$  (1): Suppose that $\xi\in \phi^n(Y)\cap Y$ for some
integer $n\not=0.$ If $n\le -1,$ then
$\xi\in Y_1\cap Y_2.$ If $n\ge 1,$ then
$\phi^{-n}(\xi)\in Y\cap \phi^{-n}(Y)\subset Y_1\cap Y_2.$

\end{proof}

\section{Traces on generalized universal irrational rotation $C^*$-algebras}

We will continue to study the traces on $A_{\theta, \gamma}.$ Here, again,
$\gamma\in C({\mathbb T})$ is a positive function and $Y$ is the set of zeros of $\gamma.$
The proof of Lemma~\ref{L: U trace} indeed implies the following
result.
\begin{Proposition}\label{P:U trace}
If $\mu$ is a complex regular Borel measure on $\mathbb{T}$ which
satisfies that
\begin{equation}\label{E:trace}
\int_{\mathbb{T}}f(e^{-2\pi i\theta}
z)d\mu(z)=\int_{\mathbb{T}}f(z)d\mu(z)
\end{equation}
for all $f(z)$ in $\widetilde{C_0({\mathbb T}\setminus Y{)}},$
the unitization of $C_0({\mathbb T}\setminus Y),$
and let $\sigma(f)=\int_{\mathbb{T}}f(z)d\mu(z)$ for $f(z)\in
C(\mathbb{T})$, then $\sigma\cdot \Phi$ is a bounded tracial
linear functional on $A_{\theta,\gamma}$. Conversely,
every bounded tracial linear functional on
$A_{\theta,\gamma}$ is given in this way.
\end{Proposition}
\begin{proof}
If $\mu$ satisfies (\ref{E:trace}) for all $f(z)$ in $\widetilde{C_0({\mathbb T}\setminus Y{)}}$,  then by a
similar argument of the proof of Lemma~\ref{L: U trace},
$\sigma\cdot \Phi$ is a bounded tracial linear functional on
$A_{\theta,\gamma}$. Conversely, suppose $\sigma$ is a
bounded tracial linear functional on
$A_{\theta,\gamma}$. By Proposition~\ref{P:U Phi},
$\sigma(a)=\sigma(\Phi(a))$. By the Riesz representation theorem,
$\sigma$ induces a complex regular Borel measure $\mu$ on
$\mathbb{T}$.

Since for all $f(z)\in C(\mathbb{T})$,
\[
\sigma(xf(w)\overline{f(w)}x^*)=\sigma(|f|^2(e^{-2\pi
i\theta}w)\gamma(e^{-2\pi
i\theta}w))\]\[=\int_{\mathbb{T}}|f|^2(e^{-2\pi
i\theta}z)\gamma(e^{-2\pi
i\theta}z)d\mu(z)=\int_{\mathbb{T}}|f|^2(z)\gamma(z)d\mu(e^{2\pi
i\theta}z)
\]
and
\[
\sigma(\overline{f(w)}x^*xf(w))=\sigma(|f|^2(w)\gamma(w))=\int_{\mathbb{T}}|f|^2(z)\gamma(z)d\mu(z),
\]
we have
\[
\int_{\mathbb{T}}|f|^2(z)\gamma(z)d\mu(e^{2\pi
i\theta}z)=\int_{\mathbb{T}}|f|^2(z)\gamma(z)d\mu(z),\quad\forall
f(z)\in C(\mathbb{T}).
\]
Since every continuous function is a linear combinations of
positive functions,
\[
\int_{\mathbb{T}}f(z)\gamma(z)d\mu(e^{2\pi
i\theta}z)=\int_{\mathbb{T}}f(z)\gamma(z)d\mu(z),\quad\forall
f(z)\in C(\mathbb{T}).
\]
This implies that (\ref{E:trace}) is true for all $f(z)\in
\overline{\gamma(z)C(\mathbb{T})}=C_0(\mathbb{T}\setminus Y)$. Since (\ref{E:trace}) is true
for $f(z)\equiv 1$, $\mu$ is a regular Borel measure on
$\mathbb{T}$ which satisfies
\[
\int_{\mathbb{T}}f(e^{-2\pi i\theta}
z)d\mu(z)=\int_{\mathbb{T}}f(z)d\mu(z)
\]
for all $f(z)$ in $\widetilde{C_0({\mathbb T}\setminus Y{)}}.$

\end{proof}

Recall that $\phi: {\mathbb T}\to {\mathbb T}$ is the rotation of circle by $\theta$, i.e., $\phi(z)=e^{2\pi i\theta}z$ for $z\in\mathbb{T}$.

\begin{Theorem}\label{P:U unique trace}
Let $Y$ be the set of zero points of $\gamma(z)$. Then the following conditions are equivalent:
\begin{enumerate}
\item There exists a unique trace on $A_{\theta,\gamma}$;
\item $\phi^n(Y)\cap Y=\emptyset$ for all integers $n\not=0$;
\item For each $\xi\in {\mathbb T},$ $Orb(\xi)\cap Y$ contains at most one point.
\end{enumerate}
\end{Theorem}
\begin{proof}
The equivalence of 2 and 3 follows from Proposition~\ref{NP}.

``$1\Rightarrow 2$". Suppose that $\phi^k(Y)\cap Y\neq\emptyset$ for some integers $k\not=0$.
Assume that $z_1\in Y$ and $z_2=\phi^{{k}}(z_1)=e^{2\pi ik\theta}z_1\in Y$. By symmetry, we may assume
that $k>0$. Let
\[
\mu=\frac{\delta_{e^{2\pi i\theta}z_1}+\delta_{e^{2\pi
i2\theta}z_1}+\cdots+\delta_{z_2}}{k},
\]
{where $\dt_t$ is the point-mass at $t.$}
Then $\widetilde{ C_0({\mathbb T}\setminus Y{)}}\subseteq \{f\in C(\mathbb{T}):\, f(z_1)=f(z_2)\}$. Note that for
$f(z)\in C(\mathbb{T})$ with $f(z_1)=f(z_2)$ we have
\[
\int_{\mathbb{T}}f(e^{-2\pi
i\theta}z)d\mu(z)=\frac{f(z_1)+f(e^{2\pi
i\theta}z_1)+\cdots+f(e^{2\pi
i(k-1)\theta})}{k}\]\[=\frac{f(e^{2\pi
i\theta}z_1)+\cdots+f(e^{2\pi
i(k-1)\theta})+f(z_2)}{k}=\int_{\mathbb{T}}f(z)d\mu(z).
\]
By Proposition~\ref{P:U trace}, $\mu$ induces a trace different
from the trace given in Lemma~\ref{L: U trace}.

``$2\Rightarrow 1$" Let $C=\widetilde{C_0({\mathbb T}\setminus Y{)}}$ be the unitization of
$C_0({\mathbb T}\setminus Y),$  and let $\rho$ be a tracial state on $A_{\theta, \gamma}.$ It follows from \ref{P:U trace} that $\rho=\mu\circ \Phi,$ where $\mu$ is a Borel probability measure on ${\mathbb T}$ such that
\beq\label{Unitrace-1}
\int_{{\mathbb T}}f(\phi^{-1}(z))d\mu(z)=\int_{{\mathbb T}} f(z)d\mu(z)
\eneq
for all $f\in C.$
Define $X_0=Y$ and $X_n=\phi^n(Y),$ $n=\pm1, \pm 2,....$
By the assumption, $\{X_n: n\in {\mathbb Z}\}$ are mutually disjoint closed subsets
of ${\mathbb T}.$
We claim that
\beq\label{Unitrace-1+1}
\mu(X_n)=0,\,\,\,n\in\mathbb{Z}.
\eneq
Let $k\ge 1$ be an integer. One can find an open subset $G\subset {\mathbb T}$ such that
\beq\label{Unitrace-1+2}
X_0\subset G\andeqn \phi^{j}(G)\cap \phi^{i}(G)=\emptyset
\eneq
if $i\not=j$ and $-k\le i, j\le k.$ Define $0\le g\le 1$ in $C({\mathbb T})$ such that $g(z)=0$ if $z\in X_0$ and $g(z)=1$ if $z\in {\mathbb T}\setminus G.$ Then $g\in C_0({\mathbb T}\setminus Y).$ Let $h=1-g.$  Then $h(z)=1$ if $z\in X_0$ and
$h(z)=0$ if $z\in {\mathbb T}\setminus G.$ Moreover, $h\in C.$
Let $h_j=h\circ \phi^{-j},$ $-k\leq j\leq k.$ Note that  $h_j(z)=1$ if $z\in \phi^j(X_0)$ and $h_j(z)=0$ if $z\in {\mathbb T}\setminus \phi^j(G)$ for $-k\leq j\leq k.$ In particular, if $-k\leq j\leq k$ and $j\neq 0$, then $h_j(z)=0$ for $z\in X_0$.   Therefore $h_j\in C_0({\mathbb T}\setminus Y)\subset C$ for $-k\leq j\leq k$ and $j\neq 0$. It follows from \ref{Unitrace-1} that
\beq\label{Unitrace-1+3}
\int_{\mathbb T}h_jd\mu=\int_{\mathbb T} hd\mu,\,\,\,-k\leq j\leq k.
\eneq
Since $h_j$ has disjoint support, (\ref{Unitrace-1+3}) implies that
\beq\label{Unitarce-1+4}
0\le \int_{\mathbb T} h_j d\mu=\int_{\mathbb T} h d\mu <\frac{1}{2k+1}.
\eneq
Therefore,
\beq\label{Unitrace-1+5}
\mu(X_j)<\frac{1}{2k+1}, \,\,\,-k\leq j\leq k.
\eneq
Since (\ref{Unitrace-1+5}) holds for any integer $k\ge 1,$ we conclude
that the claim (\ref{Unitrace-1+1}) holds.

Let $f\in C({\mathbb T})$ and  let  $\ep>0.$ Since $\mu(X_0)=\mu(X_{-1})=0$, we can
choose an open subset $O\subset {\mathbb T}$ such that
\beq\label{Unitrace-3}
Y\subset O, \mu(O)<\ep/(2\|f\|+1)\andeqn  \mu(\phi^{-1}(O))<\ep/(2\|f\|+1).
\eneq
Define a continuous function $g_1\in C$ such that $0\le g_1\le 1,$
\beq\label{Unitrace-4}
g_1(z)=0,\,\,\, {\rm when}\,\,\, z\in Y\andeqn g_1(z)=1\,\,\,{\rm when}\,\,\, z\in {\mathbb T}\setminus O.
\eneq
Note that $fg_1\in C.$ In particular,
\beq\label{Untrace-5}
\int_{\mathbb T} fg_1\circ \phi^{-1} d\mu=\int_{\mathbb T}fg_1d\mu.
\eneq
Then
\beq\label{Unitrace-6}
&&\left|\int_{{\mathbb T}} f(e^{-2i\pi \theta}z)d\mu(z) -\int_{\mathbb T}f(z)d\mu(z)\right|\\
&\le & \left|\int_{\mathbb T} (f-g_1f)\circ \phi^{-1} d\mu\right|+\left|\int_{\mathbb T} (fg_1-fg_1\circ \phi^{-1})d\mu\right|\\
&&+\left|\int_{\mathbb T} (f-g_1f) d\mu\right|\\
&\leq & \int_{\mathbb T} |(f-g_1f)\circ \phi^{-1}| d\mu+
\int_{\mathbb T} |f-g_1f| d\mu \\
&\le &\|f\| \mu(\phi^{-1}(O))+\|f\|\mu(O)|<\ep/2+\ep/2=\ep
\eneq
It follows that
\beq\label{Unitrace-7}
\int_{\mathbb T} f(e^{-2 i \pi \theta}z)d\mu(z)=\int_{\mathbb T} f(z)d\mu(z)
\eneq
for all $f\in C({\mathbb T}).$ Therefore, $\mu$ is the Haar measure on $\mathbb T.$ This shows that $A_{\theta, \gamma}$ has a unique tracial state.
\end{proof}

\begin{Remark}
\emph{If $\gamma(z)$ has a single zero point, then there exists a unique tracial state on $A_{\theta,\gamma}$}
\end{Remark}

\begin{Remark}
\emph{If $\gamma(z)$ has two zero points $z_1,z_2$, then there exists a unique tracial state on $A_{\theta,\gamma}$ if and only if
there does not exist $k\in \mathbb{N}$ such that $z_2=e^{2\pi ik\theta}z_1$.}
\end{Remark}

For a $C^*$-algebra $\mathfrak{A}$, we denote by
$\rm{Tr}(\mathfrak{A})$ the space of bounded tracial linear
functionals on $\mathfrak{A}$. Denote by $T({\mathfrak{A}})$ the tracial state space of ${\mathfrak{A}}.$

Let $\Delta$ be a subset of ${\mathbb T}$ which contains exactly one point of each orbit ${\rm Orb}(\xi)$ and let $Y$ be the set of zeros of $\gamma(z)$.

\begin{Lemma}\label{Lnumbertrace}
Let $\xi_1, \xi_2,...,\xi_r\in \Delta$ and
$Y_j=Y\cap Orb(\xi_j),$ $j=1,2,...,r.$
Let $Y_j'\subset Y_j$ be a finite subset of $Y_j$ and
let  $|Y_j'|$ be
the cardinality of $Y_j',$
Then ${\rm dim}{\rm Tr}(A_{\theta,\gamma})\ge 1+\sum_{j=1}^r (|Y_j'|-1).$
\end{Lemma}

\begin{proof}
Suppose that
$Y_j'=\left\{z_{j,1}, (e^{2\pi im_{j,1}\theta})z_{j,1},\cdots,
(e^{2\pi im_{j,n_j}\theta})z_{j,1}\right\}$ with
$1<m_{j,1}<\cdots<m_{j,n_j},$ where $|Y_j'|=n_j+1,$ $j=1,2,...,r.$
  As in the proof of Proposition~\ref{P:U unique trace}, the Haar measure $m$ together with  \[\mu_{j,1}=\frac{\delta_{\left(e^{2\pi
i\theta}\right)z_{j,1}}+\cdots+\delta_{\left(e^{2\pi
im_{j,1}\theta}\right)z_{j,1}}}{m_{j,1}-1}\]
\[\mu_{j,2}=\frac{\delta_{\left(e^{2\pi
i(m_{j,1}+1)\theta}\right)z_{j,1}}+\cdots+\delta_{\left(e^{2\pi
im_{j,2}\theta}\right)z_{j,1}}}{m_{j,2}-m_{j,1}}\]\[\vdots
\]
\[\mu_{j,n_j}=\frac{\delta_{\left(e^{2\pi
i(m_{j,n_j-1}+1)\theta}\right)z_{j,1}}+\cdots+\delta_{\left(e^{2\pi
im_{j,n_j}\theta}\right)z_{j,1}}}{m_{n_j}-m_{n_j-1}}\] induce
$1+\sum_{j=1}^r(|Y_j'|-1)$ linearly independent tracial states on
$A_{\theta,\gamma}.$  This proves that ${\rm dim}({\rm
Tr}(A_{\theta,\gamma}))\geq 1+\sum_{j=1}^r(|Y_j'|-1)$.
\end{proof}

\begin{Corollary}
Let $\xi\in {\mathbb T}$ and let $N(\xi)$  be the number of points in $Y\cap Orb(\xi)$.
If $\sum_{\xi\in \Delta}N(\xi)=\infty$, then $A_{\theta,\gamma}$ has infinitely many extreme points in its tracial state space ${\rm T}(A_{\theta,\gamma})$ and ${\rm dim}({\rm Tr}(A_{\theta, \gamma}))=\infty.$
\end{Corollary}
\begin{proof}
For any integer $N\ge 1,$ since
$\sum_{\xi\in \Delta}N(\xi)=\infty,$ one can find $\xi_1,\xi_2,...,\xi_n\in {\mathbb T}$ and finite subsets $Y_j\subset Y\cap Orb(\xi_j),$ $j=1,2,...,n$ such that
$$
\sum_{j=1}^n (|Y_j|-1)>N.
$$
It follows from \ref{Lnumbertrace} that ${\rm dim}({\rm Tr}(A_{\theta,\gamma}))\ge N.$ It follows that ${\rm dim}({\rm Tr}(A_{\theta,\gamma}))=\infty.$ The corollary follows.
\end{proof}

\begin{Proposition}\label{T:dimsion}
Let $\xi_1, \xi_2,...,\xi_r\in \Delta$ and
$Y_j=Y\cap Orb(\xi_j),$ $j=1,2,...,r$, such that
$Y=\cup_{j=1}^r Y_j.$ Suppose that $Y$ is a finite set.
 Then ${\rm
dim}({\rm
Tr}(A_{\theta,\gamma}))=1+\sum_{j=1}^r(|Y_j|-1)$, where
$|Y_j|$ is the number of elements in $Y_j$.
\end{Proposition}
\begin{proof}
By Lemma~\ref{Lnumbertrace}, ${\rm dim}({\rm
Tr}(A_{\theta,\gamma}))\geq 1+\sum_{j=1}^r(|Y_j|-1)$. We need to show
${\rm dim}({\rm
Tr}(A_{\theta,\gamma}))\leq 1+\sum_{j=1}^r(|Y_j|-1)$.
By Proposition~\ref{P:U trace}, a regular Borel
probability measure $\mu$ on $\mathbb{T}$ induces a trace on
$A_{\theta,\gamma}$ if and only if
\[
\int_{\mathbb{T}}f(z)d\mu(e^{2\pi i\theta}z)=\int_{\mathbb{T}}f(z)d\mu(z)
\]
for all $f(z)\in \widetilde{ C_0({\mathbb T}\setminus Y{)}}.$ Suppose that the zero points of  $\gamma(z)$ are $z_1,\cdots,
z_n$. Then the norm closure of $\gamma(z)C(\mathbb{T})$ in
$C(\mathbb{T})$ is
\[J=\{f(z)\in C(\mathbb{T}):\, f(z_1)=\cdots=f(z_n)=0\}\] and so
\[C=\widetilde{ C_0({\mathbb T}\setminus Y{)}}=\{f(z)\in C(\mathbb{T}):\, f(z_1)=\cdots=f(z_n)\}\subseteq C(\mathbb{T}).\] Therefore, $\mu$
induces a trace on $A_{\theta,\gamma}$ if and only if
\[
\int_{\mathbb{T}}f(z)d\mu(e^{2\pi
i\theta}z)=\int_{\mathbb{T}}f(z)d\mu(z)
\]
for all $f(z)\in C$.

Let $C^\perp=\{\rho:\, \rho\in
C(\mathbb{T})^*\,\text{and}\, \rho(a)=0 \,\text{for
all}\, a\in C\}$. Note that
$C(\mathbb{T})/C\cong \mathbb{C}^{n-1}$.
So ${\rm dim}C^\perp=n-1$. Suppose that
$Y_j=\left\{z_{j,1}, (e^{2\pi im_{j,1}\theta})z_{j,1},\cdots,
(e^{2\pi im_{j,n_j}\theta})z_{j,1}\right\}$ with
$1<m_{j,1}<\cdots<m_{j,n_j}$. Define $\mu_{j,k}'s$ as in the proof Lemma~\ref{Lnumbertrace} and let
$\nu_j=\delta_{z_{j,1}}-\delta_{z_{j+1,1}}$ for $1\leq j\leq r-1$.
Then
\[\{\mu_{j,k}-\mu_{j,k}(e^{2\pi i\theta}\cdot):\, 1\leq j\leq r, 1\leq k\leq
n_j\}\cup\{\nu_j:\,1\leq j\leq r-1\}
\]
are $n-1$ linearly independent elements in $C^\perp$.
Therefore, there are real numbers $s_{j,k}$ and $t_j$ such that
\[
\mu(e^{2\pi i\theta}E)-\mu(E)-\sum s_{j,k}(\mu_{j,k}(e^{2\pi
i\theta}E)-\mu_{j,k}(E))=\sum_{j=1}^{r-1}t_j\nu_{j}(E)
\] for all Borel measurable subset $E$ of $\mathbb{T}$.
Let $\bar{\mu}(E)=\mu(E)-\sum s_{j,k}\mu_{j,k}(E)$. Then
$\bar{u}(\{z_{1,1}\})=\mu(\{z_{1,1}\})\geq 0$ and
\begin{equation}\label{E:bar mu}
\bar{\mu}(e^{2\pi i
\theta}E)-\bar{\mu}(E)=\sum_{j=1}^{r-1}t_j(\delta_{z_{j,1}}(E)-\delta_{z_{j+1,1}}(E)).
\end{equation}
Claim $t_1=\cdots=t_{r-1}=0$. Otherwise,  we may assume that
$t_1>0$. In (\ref{E:bar mu}) let $E=\{z_{1,1}\}$, then we have
$\delta_{z_{j,1}}(E)=0$ for all $2\leq j\leq r$. Hence,
$\bar{\mu}(\{e^{2\pi i \theta}z_{1,1}\})\geq
t_1+\bar{\mu}(\{z_{1,1}\})\geq t_1$. In (\ref{E:bar mu}) let
$E=\{e^{2\pi i\theta}z_{1,1}\}$, then we have
$\delta_{z_{j,1}}(E)=0$ for all $1\leq j\leq r$. Hence,
$\bar{\mu}(\{e^{2\pi i2\theta}z_{1,1}\})=\bar{\mu}(\{e^{2\pi
i\theta}z_{1,1}\})\geq t_1>0$. By induction, we have
$\bar{\mu}(\{e^{2\pi in\theta}z_{1,1}\})\geq t_1>0$ for all
$n\in\mathbb{N}$. This contradicts to the fact that $\bar{\mu}$ is
a bounded real measure.

Therefore,
\[
\mu(e^{2\pi i\theta}E)-\mu(E)=\sum s_{j,k}(\mu_{j,k}(e^{2\pi
i\theta}E)-\mu_{j,k}(E))
\] for all Borel measurable subset $E$ of $\mathbb{T}$, i.e.,
\[
\mu(e^{2\pi i\theta}E)-\sum s_{j,k}(\mu_{j,k}(e^{2\pi
i\theta}E))=\mu(E)-\sum s_{j,k}(\mu_{j,k}(E))
\]
for all Borel subset $E$ of $\mathbb{T}$. Let
\[
\nu=\mu-\sum s_{j,k}\mu_{j,k}.
\]
Then
\[
\nu(e^{2\pi i\theta}E)=\nu(E)
\]
for all Borel subsets of $\mathbb{T}$.  Therefore, for
every $n\in \mathbb{N}$, $\nu(e^{2\pi in\theta}E)=\nu(E)$ for all
Borel subsets $E$ of $\mathbb{T}$. Since $\theta$ is an irrational
number, $\{e^{2\pi in\theta}:\,n\in\mathbb{N}\}$ is dense in
$\mathbb{T}$. By the Lebesgue dominated theorem, $\nu(zE)=\nu(E)$
for all Borel subsets $E$ of $\mathbb{T}$ and $z\in \mathbb{T}$.
By the uniqueness of the Haar measure on $\mathbb{T}$,  there
exists $t\in\mathbb{R}$ such that $\nu=tm$, i.e., $\mu=\sum s_{j,k}\mu_{j,k}+tm$
This implies that ${\rm dim}({\rm Tr}(A_{\theta,\gamma}))\leq
1+\sum_{j=1}^r(|Y_j|-1)$. So ${\rm dim}({\rm
Tr}(A_{\theta,\gamma}))=1+\sum_{j=1}^r(|Y_j|-1)$.
\end{proof}

\begin{Proposition}\label{T:extrem points}
Suppose $\gamma(z)$ has finitely many points in its zero set $Y$ and there are $\xi_1, \xi_2,...,\xi_r\in \Delta$ such that
$Y=\cup_{j=1}^r Y_j,$ where $Y_j=Y\cap Orb(\xi_j).$
Then
$\tau$ and the traces induced by $\mu_{j,k}'s$ constructed in
Lemma~\ref{Lnumbertrace} are precisely the extreme points of
${\rm T}(A_{\theta,\gamma})$.
\end{Proposition}
\begin{proof}
Let $\sigma$ be a tracial state on $A_{\theta,\gamma}$
induced by a regular Borel probability measure $\mu$ on
$\mathbb{T}$. Then by the proof of Proposition~\ref{T:dimsion}, there are real numbers $t,s_{j,k}$ such that
\[
\mu(E)=tm(E)+\sum s_{j,k}\mu_{j,k}(E)
\]
for all Borel subsets $E$ of $\mathbb{T}$. Since $m$ and
$\mu_{j,k}'$ are mutually disjoint measures, $t,s_{j,k}\geq 0$ and
$t+\sum s_{j,k}=1$. This shows that $\tau$ and the traces induced
by $\mu_{j,k}'s$ constructed in Lemma~\ref{Lnumbertrace} are
precisely the extreme points of
${\rm T}(A_{\theta,\gamma})$.
\end{proof}

\begin{Corollary}\label{C:extreme points}
Suppose $\gamma(z)$ has finite zero points. Then $\tau$ is the
unique extreme point in ${\rm T}(A_{\theta,\gamma})$
which is faithful on $A_{\theta,\gamma}$.
\end{Corollary}

\section{Simplicity of generalized universal $C^*$-algebras}

In this section, we provide a characterization of simplicity of a
generalized universal algebra $A_{\theta,\gamma}$ in
terms of the zero points of $\gamma(z)$.
We begin with the following lemma.

\begin{Lemma}\label{L:less than 1 general case}
Let $f_n(z)\in C(\mathbb{T})$ for $-M\leq n\leq N$. Then
\[
\|x^kf_k(w)\|\leq\left\|\sum_{n=1}^Nx^nf_n(w)+f_0(w)+\sum_{m=1}^Mf_{-m}(w)\left(x^*\right)^{m}\right\|,
\]
and
\[
\|f_{-k}(w)(x^*)^k\|\leq\left\|\sum_{n=1}^Nx^nf_n(w)+f_0(w)+\sum_{m=1}^Mf_{-m}(w)\left(x^*\right)^{m}\right\|.
\]
 \end{Lemma}

\begin{proof}
There is a function $\gamma_k\in C({\mathbb T})_+$ such that
$x^k=u^k \gamma_k(w),$ $k=1,2,...,$
Therefore $u^{-k}x^kf_k(w)=\gamma_k(w)f_k(w).$

Put $a=\sum_{i=0}^N x^if_i(w)+\sum_{j=1}^Mf_{-j}(w)(x^*)^k.$
Let $\Phi$ be the conditional expectation. Then
\beq\nonumber
\|x^kf_k(w)\|=\|u^{-k}x^kf_k(w)\|=\|\Phi(u^{-k}a)\|\le \|u^{-k}a\|=\|a\|.
\eneq
So the first part of the lemma follows. The second part follows similarly.
\end{proof}

\begin{Lemma}\label{T:simple}
Let $Y_1$ be the set of zero points of functions
$\gamma(e^{2\pi in\theta}z)$ for $n\geq 0$, and let $Y_2$ be
the set of zero points of functions $\gamma(e^{-2\pi in\theta}z)$
for $n\geq 1$. Then $A_{\theta,\gamma}$ is a simple
algebra if and only if $Y_1\cap Y_2=\emptyset$.
\end{Lemma}
\begin{proof}
Suppose $Y_1\cap Y_2=\emptyset$ and $J$ is a non-zero
ideal of $A_{\theta,\gamma}$. Then there is a positive
nonzero element $x$ in $J$. Since $w^jx(w^*)^j\in J$, the limit
formula for $\Phi(x)$ in Proposition~\ref{P:U Phi} shows that $\Phi(x)\in J\cap C^*(w)$. Since
$\Phi$ is faithful, $\Phi(x)>0$. So $J\cap C^*(w)$ is a nontrivial
ideal in $C^*(w)$, which is contained in a maximal nontrivial
ideal
\[
I=\{f(w)|f(z)\in C(\mathbb{T})\,\text{and}\,f(z_0)=0\,\text{for
some}\,z_0\in \mathbb{T}\}
\]
of $C^*(w)$.

Let $f(z)\in C(\mathbb{T})$ such that $f(w)\in J\cap C^*(w)\subset
I$. Then $f(z_0)=0$. By (\ref{U5-f(z)}) and (\ref{E:x*rxr}), we have
\begin{equation}\label{E:x*gx}
x^*f(w)x=f(e^{2\pi i\theta}w)\gamma(w)\in J\cap C^*(w)\subset I.
\end{equation}
By (\ref{U4-f(z)}) and (\ref{E:xrx*r}), we have
\begin{equation}\label{E:xgx*}
xf(w)x^*=f(e^{-2\pi i\theta}w)\gamma(-e^{2\pi i\theta}w)\in J\cap
C^*(w)\subset I.
\end{equation}

Case 1. Suppose $z_0\in Y_2$. Then the assumption of the
theorem implies that $z_0\notin Y_1$. So (\ref{E:x*gx})
implies that $f(e^{2\pi i\theta}z_0)=0$. Repeat using
(\ref{E:x*gx}), we have for all $n\in \mathbb{N}$,
\[
(x^*)^nf(w)x^n=f(e^{2\pi in\theta}w)\gamma(e^{2\pi
i(n-1)\theta}w)\gamma(e^{2\pi i(n-2)\theta}w)\cdots \gamma(w)\in
J\cap C^*(w)\subset I.
\]
Thus $f(e^{2\pi in\theta}z_0)=0$ for all $n\in \mathbb{N}$. Since
$\{e^{2\pi in\theta}z_0:\, n\in\mathbb{N}\}$ is dense in
$\mathbb{T}$, $f(z)=0$ for all $z\in \mathbb{T}$. This implies
that $J\cap C^*(w)$ is trivial and we obtain a contradiction.

Case 2. Suppose $z_0\notin Y_2$.  Then (\ref{E:xgx*})
implies that $f(e^{-2\pi i\theta}z_0)=0$. Repeat using
(\ref{E:xgx*}), we have for all $n\in \mathbb{N}$,
\[
x^nf(w)(x^*)^n=f(e^{-2\pi in\theta}w)\gamma(e^{-2\pi
in\theta}w)\gamma(e^{-2\pi i(n-1)\theta}w)\cdots \gamma(e^{-2\pi
i\theta}w)\in J\cap C^*(w)\subset I.
\]
Thus $f(e^{-2\pi in\theta}z_0)=0$ for all $n\in \mathbb{N}$. Since
$\{e^{-2\pi in\theta}z_0:\, n\in\mathbb{N}\}$ is dense in
$\mathbb{T}$, $f(z)=0$ for all $z\in \mathbb{T}$. This implies
that $J\cap C^*(w)$ is trivial and we obtain a contradiction.
\vskip 0.5cm

Conversely, suppose $Y_1\cap Y_2\neq\emptyset$. We may
assume that $\lambda\in \mathbb{T}$ is a zero point of
$\gamma(e^{2\pi in\theta}z)$ and $\gamma(e^{-2\pi im\theta}z)$.
Consider the subset
\[
J=\{f(w)|f(z)\in C(\mathbb{T})\,\text{and}\, f(e^{2\pi i
n\theta}\lambda)=\cdots=f(\lambda)=\cdots=f(e^{-2\pi i
m\theta}\lambda)=0\}
\]
of $C^*(w)$. Claim that
$I=A_{\theta,\gamma}JA_{\theta,\gamma}$ is a
two-sided ideal of $A_{\theta,\gamma}$. Otherwise, there
exists $f_i(w)\in J$,
\[
a_i=\sum_{n=1}^K (x^*)^ng^i_{-n}(w)+g^i(w)+\sum_{n=1}^Kg^i_{n}(w)x^n,
\]
and
\[
b_i=\sum_{n=1}^K (x^*)^nh^i_{-n}(w)+h^i(w)+\sum_{n=1}^Kh^i_{n}(w)x^n,
\]
for sufficiently large $K\in \mathbb{N}$ such that
\[
\left\|\sum_{i=1}^Na_if_i(w)b_i-1\right\|<1,
\]
where $g^i_n,g^i,h^i_n,h^i\in C(\mathbb{T})$. By Lemma~\ref{L:less than 1 general case}
and simple computations, we have
\[
\|\sum_{i=1}^N g^i_{-K}(e^{2\pi i K\theta}w)f_i(e^{2\pi i K\theta}w)h^i_{K}(e^{2\pi i
K\theta}w)\gamma(e^{2\pi i (K-1)\theta}w)\cdots
\gamma(w)+\]\[g^i_{-(K-1)}(e^{2\pi i (K-1)\theta}w)f_i(e^{2\pi i
(K-1)\theta}w)h^i_{K-1}(e^{2\pi i (K-1)\theta}w)\gamma(e^{2\pi i
(K-2)\theta}w)\cdots \gamma(w) +\cdots+
\]
\[
g^i_{-1}(e^{2\pi i\theta}w)f_i(e^{2\pi i\theta}w)h^i_{1}(e^{2\pi
i\theta}w)\gamma(w)+g^i(w)f_i(w)h^i(w)+g^i_1(w)f_i(e^{-2\pi
i\theta}w)h^i_{-1}(w)\gamma(e^{-2\pi i\theta}w)+\cdots+
\]
\[
g^i_{K-1}(w)f_i(e^{-2\pi i(K-1)\theta}w)h^i_{-(K-1)}(w)\gamma(e^{-2\pi
(K-1)i\theta}w)\cdots \gamma(e^{-2\pi i\theta}w)+
\]
\[
g^i_{K}(w)f_i(e^{-2\pi iK\theta}w)h^i_{-K}(w)\gamma(e^{-2\pi
Ki\theta}w)\cdots \gamma(e^{-2\pi i\theta}w)-1\|<1.
\]
Let
\[
\bar{f}(z)=\sum_{i=1}^Ng^i_{-K}(e^{2\pi i K\theta}z)f_i(e^{2\pi i
K\theta}z)h^i_{K}(e^{2\pi i K\theta}z)\gamma(e^{2\pi i
(K-1)\theta}z)\cdots \gamma(z)+\cdots+
\]
\[
g^i_{-1}(e^{2\pi i\theta}z)f_i(e^{2\pi i\theta}z)h^i_{1}(e^{2\pi
i\theta}z)\gamma(z)+g^i(z)f_i(z)h^i(z)+g^i_1(z)f_i(e^{-2\pi
i\theta}z)h^i_{-1}(z)\gamma(e^{-2\pi i\theta}z)+\cdots+
\]
\[
g^i_{K}(z)f_i(e^{-2\pi iK\theta}z)h^i_{-K}(z)\gamma(e^{-2\pi
Ki\theta}z)\cdots \gamma(e^{-2\pi i\theta}z).
\]

Since $f_i(z)\in J$, $f_i(e^{2\pi i
n\theta}\lambda)=\cdots=f_i(\lambda)=\cdots=f_i(e^{-2\pi i
m\theta}\lambda)=0$. Note that $\gamma(e^{2\pi
in\theta}\lambda)=\gamma(e^{-2\pi im\theta}\lambda)=0$. So
$\bar{f}(\lambda)=0$. Hence $\|\bar{f}(z)-1\|\geq 1$ and
$\|\bar{f}(w)-1\|\geq 1$. By Lemma~\ref{L:less than 1 general case},
\[
\left\|\sum_{i=1}^Na_if_i(w)b_i-1\right\|\geq \|\bar{f}(w)-1\|\geq 1.
\]
 This is a contradiction.
\end{proof}

\begin{Theorem}\label{SimpleAtheta}
{Let $\theta$ be an irrational number, $\gamma\in C({\mathbb T})$ be a non-negative function, let $Y$ be the set of zeros of $\gamma$ and let $\phi: {\mathbb T}\to {\mathbb T}$ be the homeomorphism by rotation of $\theta.$
Then the following are equivalent:

{(1) $A_{\theta, \gamma}$ is simple;}

{(2) $A_{\theta, \gamma}$ has a unique tracial state;}

{(3) $\phi^n(Y)\cap Y=\emptyset$ for all integers $n\not=0.$ }

{(4) For each $\xi\in {\mathbb T},$ $Orb(\xi)\cap Y$ contains at most one point. }}
\end{Theorem}
\begin{proof}
The equivalence of (2), (3) and (4) follows from Theorem~\ref{P:U unique trace}.  Let $Y_1$ be the set of zero points of functions
$\gamma(e^{2\pi in\theta}z)$ for $n\geq 0$, and let $Y_2$ be
the set of zero points of functions $\gamma(e^{-2\pi in\theta}z)$
for $n\geq 1$. By Proposition~\ref{NP}, condition (3) is equivalent to $Y_1\cap Y_2=\emptyset$. By Lemma~\ref{T:simple}, (1) is equivalent to (3).
\end{proof}

\begin{Corollary}\label{C:simplicity}
Suppose $\gamma(z)\in C(\mathbb{T})$ is a positive function with a
single zero point. Then $A_{\theta,\gamma}$ is a simple $C^*$-algebra with a unique tracial state.
\end{Corollary}

\begin{Corollary}\label{C:simplicity2}
Suppose $\gamma(z)\in C(\mathbb{T})$ is a positive function with two
 zero points $z_1,z_2$.  Then $A_{\theta,\gamma}$ is a simple $C^*$-algebra with a unique tracial state if and only if there does not exist
 integer $k$ such that $z_2=e^{2\pi ik\theta}z_1$.
\end{Corollary}

\begin{Corollary}\label{C:non-simplicity}
If $m(\{z|\gamma(z)=0\})>0$, then $A_{\theta,\gamma}$ is
not simple.
\end{Corollary}
\begin{proof}
Let $Y=\{z|\gamma(z)=0\}$.  If $A_{\theta, \gamma}$ is  simple,
then by Theorem \ref{SimpleAtheta}, $\phi^n(Y)\cap Y=\emptyset$
for all integers $n\not=0.$ Then $\{\phi^n(Y): n\in {\mathbb Z}\}$ is a sequence of mutually disjoint subsets. Therefore $m(Y)=0.$

\end{proof}

\section{Rieffel's projections in generalized universal algebras}

\begin{Lemma}\label{L:lambda z}
If $\lambda\in \mathbb{T}$, then
$A_{\theta,\gamma(z)}\cong
A_{\theta,\gamma(\lambda z)}$.
\end{Lemma}
\begin{proof}
Let $A_{\theta,\gamma(z)}=C^*(x,w)$ and
$A_{\theta,\gamma(\lambda z)}=C^*(x',w')$.
 Then $x',\lambda w'$ satisfy (\ref{U1})-(\ref{U4}) for $\gamma(z)$. So there is a homomorphism
  $\varphi:\,A_{\theta,\gamma(z)}\rightarrow A_{\theta,\gamma(\lambda z)}$ such that $\varphi(x)=x'$, $\varphi(w)=\lambda w'$. By symmetry, there is a homomorphism $\psi:\,A_{\theta,\gamma(\lambda z)}\rightarrow A_{\theta,\gamma(z)}$ such that $\psi(x')=x$, $\psi(w')=\lambda w$. Hence $\psi\cdot \varphi(x)=x$ and $\psi\cdot\varphi(w)=w$; $\varphi\cdot\psi(x')=x'$ and $\varphi\cdot\psi(w')=w'$. So $\varphi$ is an isomorphism from  $A_{\theta,\gamma(z)}$ onto $ A_{\theta,\gamma(\lambda z)}$.
\end{proof}

\begin{Lemma}\label{L:lambda 1-theta}
$A_{\theta,\gamma}\cong A_{1-\theta,\gamma}$.
\end{Lemma}
\begin{proof}
Let $A_{\theta,\gamma}=C^*(x,w)$ and
$A_{1-\theta,\gamma}=C^*(x',w')$.
 Then $x', (w')^*$ satisfy (\ref{U1})-(\ref{U4}) for $\theta$ and $\gamma$. So there is a homomorphism $\varphi:\,A_{\theta,\gamma}\rightarrow A_{1-\theta,\gamma}$ such that $\varphi(x)=x'$, $\varphi(w)= (w')^*$. By symmetry, there is a homomorphism $\psi:\,A_{1-\theta,\gamma}\rightarrow A_{\theta,\gamma}$ such that $\varphi(x')=x$, $\varphi(w')= w^*$. Hence $\psi\cdot \varphi(x)=x$ and $\psi\cdot\varphi(w)=w$; $\varphi\cdot\psi(x')=x'$ and $\varphi\cdot\psi(w')=w'$. So $\varphi$ is an isomorphism from  $A_{\theta,\gamma}$ onto $ A_{1-\theta,\gamma}$.
\end{proof}

The proof the following theorem is similar to the proof of Theorem 1.1 of~\cite{Ri}. However, some details should be treated carefully.
\begin{Theorem}\label{T:Rieffel projection}
Suppose $\gamma$ is a positive function in $C(\mathbb{T})$ and
there exists $\lambda\in \mathbb{T}$ such that $\gamma(\lambda
e^{2\pi in\theta})\neq 0$ for all nonnegative integers $n$. Then
for every $\alpha$ in $(\mathbb{Z}+\mathbb{Z}\theta)\cap [0,1]$,
there is a projection $p$ in $A_{\theta,\gamma}$ such
that $\tau(p)=\alpha$.
\end{Theorem}
\begin{proof}
By Lemma~\ref{L:lambda z}, we may assume that $\lambda=1$. Firstly
we prove if $\alpha=\theta\in (0,1)$ then there exists a
projection $p$ in $A_{\theta,\gamma}$ such that
$\tau(p)=\theta$. By Lemma~\ref{L:lambda 1-theta}, we may assume
that $0<\theta<1/2$.

A dense set of elements of $A_{\theta,\gamma}$ can be
represented by a finite sum of the form $\sum_{i=1}^n
f_i(w)x^i+f(w)+\sum_{j=1}^mf_{-j}(w)\left(x^*\right)^j$, where
$f_k(z), f(z)\in C(\mathbb{T})$. Note that the set
$C(\mathbb{T})x^i$, $C(\mathbb{T})$,
$C(\mathbb{T})\left(x^*\right)^j$ are mutually orthogonal to each
other in $L^2(A_{\theta,\gamma}, \tau)$. In the
following we identify $C^*(w)$ with $C(\mathbb{R}/\mathbb{Z})$.
For $f(t)\in C(\mathbb{R}/\mathbb{Z})$, define
$f_{\theta}(t)=f(t-\theta)$. Let $\beta(t)=\left(\gamma(e^{2\pi
it})\right)^{1/2}$. Then $\beta(n\theta)\neq 0$
for all nonnegative integers $n$.

We look for a projection $p=g(t)x+f(t)+h(t)x^*$ such that
$\tau(p)=\theta$.  Since $p=p^*$, by (\ref{U4}) and (\ref{U5}),
\[
g(t)x+f(t)+h(t)x^*=x^*\bar{g}(t)+\bar{f}(t)+x\bar{h}(t)=
\bar{g}_{-\theta}(t)x^*+\bar{f}(t)+\bar{h}_{\theta}(t)x.
\]
By comparing coefficients, we see that $f=\bar{f}$ is a real
valued function; and that $h(t)=\overline{g(t+\theta)}$ or
equivalently $h(t-\theta)=\overline{g(t)}$. Since $p=p^2$,
(\ref{U4-f(z)})-(\ref{E:xrx*r}) imply
\[
g(t)x+f(t)+h(t)x^*=g(t)g_\theta(t)
x^2+(g(t)(f(t)+f_\theta(t)))x+\left[g(t)h_\theta(t)\beta^2(t-\theta)+f^2(t)+h(t)g_{-\theta}(t)\beta^2(t)\right]
\]
\[
+h(t)(f(t)+f_{-\theta}(t))x^*+h(t)h_{-\theta}(t)\left(x^*\right)^2.
\]
By comparing coefficients and replacing $h$'s with $g$'s using the
relation between them, we arrive at the necessary and sufficient
conditions:
\begin{equation}\label{E:g(t)g(t-theta)}
g(t)g(t-\theta)=0,
\end{equation}
\begin{equation}\label{E:g(t)(1-t)}
g(t)(1-f(t)-f(t-\theta))=0,
\end{equation}
\begin{equation}\label{E:f(t)-f(t)2}
f(t)-f(t)^2=|g(t)\beta(t-\theta)|^2+|g(t+\theta)\beta(t)|^2.
\end{equation}
 Pick any positive $\epsilon>0$ such that
$\theta+\epsilon<1/2$. Define $f$ to be the piece-wise linear
function
\[
f(t)=\begin{cases}
\epsilon^{-1}t& \text{for}\quad 0\leq t\leq \epsilon\\
1&\text{for}\quad \epsilon\leq t\leq \theta\\
\epsilon^{-1}(\theta+\epsilon-t) &\text{for}\quad \theta\leq t\leq \theta+\epsilon\\
0&\text{for}\quad \theta+\epsilon\leq t\leq 1
\end{cases}
\]
and define
\[
g(t)=\begin{cases}
\sqrt{f(t)-f(t)^2}/\beta(t-\theta)&\text{for}\quad \theta\leq t\leq \theta+\epsilon\\
0&\text{otherwise}
\end{cases}.
\]
Since $\beta(0)\neq 0$, $g(t)\in C(\mathbb{T})$ for sufficiently
small $\epsilon>0$. Then $f(t)$ and $g(t)$ satisfy equations
(\ref{E:g(t)g(t-theta)}), (\ref{E:g(t)(1-t)}), and
(\ref{E:f(t)-f(t)2}). So $\tau(p)=\int_0^1f(t)dt=\theta$. We also
get the projection $1-p$ with trace $\tau(1-p)=1-\theta$.

In the following we show that for $k\geq 2$ there is a projection
$q$ such that $\tau(q)$ is the fractional part ${k\theta}$ of
$k\theta$. Let $\alpha=\{k\theta\}$. We may assume that
$\alpha<1/2$.  The idea is similar. Let
$q=g_1(t)(u+v)^k+f_1(t)+h_1(t)\left((u+v)^*\right)^k$. Then we
will have the following equations
\begin{equation}\label{E:g(t)g(t-theta)'}
g_1(t)g_1(t-\alpha)=0,
\end{equation}
\begin{equation}\label{E:g(t)(1-t)'}
g_1(t)(1-f_1(t)-f_1(t-\alpha))=0,
\end{equation}
\begin{equation*}
f_1(t)-f_1(t)^2=|g_1(t)\beta(t-k\theta)\cdots\beta
(t-\theta)|^2+|g_1(t+\alpha)\beta(t+(k-1)\theta)\cdots\beta(t)|^2
\end{equation*}
\begin{equation*}
=|g_1(t)\beta(t-k\theta)\beta(t-(k-1)\theta)\cdots\beta(t-\theta)|^2
\end{equation*}
\begin{equation}\label{E:f(t)-f(t)2'}
+|g_1(t+\alpha)\beta((t+\alpha)-k\theta)\beta((t+\alpha)-(k-1)\theta)\cdots\beta((t+\alpha)-\theta)|^2.
\end{equation}
 Pick any positive $\epsilon>0$ such that $\theta+\epsilon<1/2$. Define $f_1$ to be the piece-wise linear function
\[
f_1(t)=\begin{cases}
\epsilon^{-1}t& \text{for}\quad 0\leq t\leq \epsilon\\
1&\text{for}\quad \epsilon\leq t\leq \alpha\\
\epsilon^{-1}(\alpha+\epsilon-t) &\text{for}\quad \alpha\leq t\leq \alpha+\epsilon\\
0&\text{for}\quad \alpha+\epsilon\leq t\leq 1
\end{cases}
\]
and define
\[
g_1(t)=\begin{cases}
\sqrt{f_1(t)-f_1(t)^2}/\beta(t-k\theta)\cdots\beta(t-\theta)&\text{for}\quad \alpha\leq t\leq \alpha+\epsilon\\
0&\text{otherwise}
\end{cases}.
\]
Since $\beta(0)\neq 0$, $\beta(\theta)\neq 0$, $\cdots$,
$\beta((k-1)\theta)\neq 0$,  $g_1(t)\in C(\mathbb{T})$ for
sufficiently small $\epsilon>0$. Then $f_1(t)$ and $g_1(t)$
satisfy equations (\ref{E:g(t)g(t-theta)'}), (\ref{E:g(t)(1-t)'}),
and (\ref{E:f(t)-f(t)2'}).
 So $\tau(q)=\int_0^1f_1(t)dt=\alpha$.

\end{proof}

\begin{Corollary}\label{C:Rieffel projections}
If $m(\{z|\gamma(z)=0\})=0$, e.g., the zero points of $f(z)$ is
countable, then for every $\alpha$ in
$(\mathbb{Z}+\mathbb{Z}\theta)\cap [0,1]$, there is a projection
$p$ in $A_{\theta,\gamma}$ such that $\tau(p)=\alpha$.
\end{Corollary}
\begin{proof}
We divide $\mathbb{T}$ into equivalent classes $F_{\alpha}$, where $x,y\in F_\alpha$
 if and only if $x=e^{2\pi ik\theta}y$ for some $k\in \mathbb{Z}$.
 Suppose $\forall \alpha$, $F_\alpha\cap \{z|\gamma(z)=0\}\neq
 \emptyset$.
 By axiom of choice we can choose a representative set
  $\{x_\alpha\}_{\alpha\in Y}$ of $\{F_\alpha\}_{\alpha\in Y}$ such that $x_\alpha\in F_\alpha\cap
  \{z|\gamma(z)=0\}$ for each $\alpha\in Y$. Then $m(\{x_\alpha\}_{\alpha\in
  Y})=0$.
   On the other hand it is well-known that
   $\{x_\alpha\}_{\alpha\in Y}$ is not Lebesgue measurable.
   This is a contradiction. Therefore, there exists $\alpha\in
   Y$ such that the intersection of $F_\alpha$ and the set of zero
points of $\gamma$ is empty. Now the corollary follows from
Theorem~\ref{T:Rieffel projection}.
\end{proof}

Combining Corollary~\ref{C:non-simplicity} and
Corollary~\ref{C:Rieffel projections}, we obtain the following
result.
\begin{Corollary}
If a generalized universal $C^*$-algebra
$A_{\theta,\gamma}$ is simple, then for every $\alpha$
in $(\mathbb{Z}+\mathbb{Z}\theta)\cap [0,1]$, there is a
projection $p$ in $A_{\theta,\gamma}$ such that
$\tau(p)=\alpha$.
\end{Corollary}

This corollary also follows from \ref{NTKtheory}.

\section{K-groups of generalized universal irrational rotation algebras}

 Let $A_\theta$ be the universal irrational rotation $C^*$-algebra with two
unitary generators $u,v$ satisfying $vu=e^{2\pi i\theta}uv$. Then there exists an action $\alpha_z$ of $\mathbb{T}$ on
$A_{\theta}$ defined by $\alpha_z(u)=zu$ and
$\alpha_z(v)=v$.
 By Theorem~\ref{T:universal property}, we may identify $A_{\theta,\gamma}$ with the
unital $C^*$-subalgebra $B$ of $A_\theta$ generated by $u\gamma^{1/2}(v)$ and $v$. Then $x=u\gamma^{1/2}(v)$ and  $w=v$.  Let $A$ be the unital
$C^*$-algebra generated by $v$. The following definition is introduced by Ruy Excel in~\cite{Ex}.

\begin{Definition}
\emph{For each $n\in \mathbb{Z}$ the $n^{th}$ spectral subspace
for $\alpha$ is defined by
\[
B_n=\{b\in
A_{\theta,\gamma}:\,\alpha_z(b)=z^nb\quad\text{for}\,
z\in \mathbb{T}\}.
\]
}
\end{Definition}

\begin{Lemma}\label{L:Bn}
$B_0=A$ and $B_1=\{uf(v):\, f(\lambda)=0\quad\text{for}\,
\lambda\in Y\}$.
\end{Lemma}
\begin{proof}
By Corollary~\ref{C:rho t}, $B_0=A$. We need to show
$B_1=\{uf(v):\, f(\lambda)=0\quad\text{for}\, \lambda\in
Y\}$. Note that $\alpha_z(xg(v))=zxg(v)$.  Since the norm
closure of $\{xg(v):\, g\in C(\mathbb{T})\}$ is $\{u f(v):\, f(\lambda)=0\quad\text{for}\,
\lambda\in Y\}$, $\{u f(v):\, f(\lambda)=0\quad\text{for}\,
\lambda\in Y\}\subseteq B_1$. On the other hand, if $b\in
B_1$, then $\alpha_z(u^*b)=u^*b$ for all $z\in \mathbb{T}$. This
implies that $b=uf(v)$ for some $f(z)\in C(\mathbb{T})$. Suppose
$f(\lambda_0)\neq 0$ for some $\lambda_0\in Y$. Then for any
$y=\sum_{n=1}^N x^nf_n(z)+f_0(z)+\sum_{n=1}^Nf_{-n}(z)(x^*)^n$, by
Lemma~\ref{L:less than 1 general case} we have
\[
\|y-z\|\geq
\|xf_1(v)-uf(v)\|=\|uh(v)f_1(v)-uf(v)\|=\|h(v)f_1(v)-f(v)\|\geq
|f(\lambda_0)|>0.
\]
Thus for any $y\in A_{\theta,\gamma}$ we have
$\|y-uf(v)\|\geq |f(\lambda_0)|>0$. This is a contradiction. So
$B_1=\{uf(v):\, f(\lambda)=0\quad\text{for}\, \lambda\in
Y\}$.
\end{proof}

\begin{Definition}
\emph{ If $X$ and $Y$ are subsets of a $C^*$-algebra, then $XY$
denotes the closed linear span of the set of products $xy$ with
$x\in X$ and $y\in Y$. }
\end{Definition}

\begin{Corollary}\label{C:B1B1*}
$B_1^*B_1=\{f(v):\, f(\lambda)=0\quad\text{for}\, \lambda\in
Y\}\subseteq A$ and $B_1B_1^*=uB_1B_1^*u^*\subset A$.
\end{Corollary}

\begin{Lemma}\label{L:regular}
The action of $\mathbb{T}$ on $A_{\theta,\gamma}$ is
regular in the sense of~\cite{Ex} \emph{(}see Definition
4.4\emph{)}, i.e., there exist an isomorphism $\theta:
B_1^*B_1\rightarrow B_1B_1^*$ and a linear isometry $\phi$ from
$B_1^*$ onto $B_1B_1^*$ such that for $y_1,y_2\in B_1$, $a\in
B_1^*B_1$ and $b\in B_1B_1^*$,
\begin{enumerate}
\item $\phi(y_1^*b)=\phi(y_1^*)b$; \item
$\phi(ay_1^*)=\theta(a)\phi(y_1^*)$; \item
$\phi(y_1^*)^*\phi(y_2^*)=y_1y_2^*$;\item
$\phi(y_1^*)\phi(y_2^*)^*=\theta(y_1^*y_2)$.
\end{enumerate}
\end{Lemma}
\begin{proof}
By Corollary~\ref{C:B1B1*}, $B_1B_1^*=uB_1B_1^*u^*$. Let
$\theta(f(v))=uf(v)u^*$. Then $\theta$ is an isomorphism of
$B_1B_1^*$ onto $B_1^*B_1$. Define $\phi(f(v)u^*)=uf(v)u^*$. By
Lemma~\ref{L:Bn} and Corollary~\ref{C:B1B1*}, $\phi$ is a linear
isometry from $B_1^*$ onto $B_1B_1^*$. Let $y_1=uf_1(v)$ and
$y_2=uf_2(v)$ such that $y_1,y_2\in B_1$, $a=g_1(v)\in
B_1^*B_1\subset A$ and $b=g_2(v)\in B_1B_1^*\subset A$. Then
\[
\phi(y_1^*b)=\phi(\bar{f}_1(v)u^*g_2(v))=\phi(\bar{f}_1(v)\theta^{-1}(g_2(v))u^*)=u\bar{f}_1(v)\theta^{-1}(g_2(v))u^*=u\bar{f}_1(v)u^*g_2(v)=\phi(y_1^*)b,
\]
\[
\phi(ay_1^*)=\phi(g_1(v)\bar{f}_1(v)u^*)=ug_1(v)\bar{f}_1(v)u^*=\theta(g_1(v))u\bar{f}_1(v)u^*=\theta(a)\phi(y_1^*),
\]
\[
\phi(y_1^*)^*\phi(y_2^*)=(uy_1^*)^*(uy_2^*)=y_1y_2^*,
\]
\[
\phi(y_1^*)\phi(y_2^*)^*=(uy_1^*)(uy_2^*)^*=uy_1^*y_2u^*=\theta(y_1^*y_2).
\]
\end{proof}

\begin{Lemma}\label{L:Ruy}
Let $\Theta=(\theta, B_1^*B_1, B_1B_1^*)$ be the partial
automorphism of the fixed point algebra $A$ as in~\cite{Ex}. Then
there exists an isomorphism
\[
\varphi: C^*(A,\Theta)\rightarrow A_{\theta,\gamma}.
\]
\end{Lemma}
\begin{proof}
Clearly $A_{\theta,\gamma}$ is generated by the fixed
point algebra $A$ and the first spectral subspace $B_1$. So the
action $\alpha$ of $\mathbb{T}$ on $A_{\theta,\gamma}$
is semi-saturated (see Definition 4.1 of ~\cite{Ex}). By
Lemma~\ref{L:regular}, $\alpha$ is also regular. By Theorem 4.21
of ~\cite{Ex}, there exists an isomorphism
\[
\varphi: C^*(A,\Theta)\rightarrow A_{\theta,\gamma}.
\]
\end{proof}

\begin{Theorem}\label{T:K-groups}
Let $Y$ be the set of zeros of $\gamma.$ If ${\mathbb T}\not=Y\neq \emptyset$, then
\beq\label{TKgroup-1}
K_1({A}_{\theta, \gamma})={\mathbb Z}
\eneq
and there exists a splitting short exact sequence:
\beq\label{TKgroup-2}
0\to {\mathbb Z}\to K_0({A}_{\theta, \gamma})\to C(Y, {\mathbb Z})\to 0.
\eneq
In particular, if $Y$ has $n$ points, then
\beq\label{TKgroup-3}
K_0({A}_{\theta, \gamma})={\mathbb Z}^{n+1}.
\eneq
\end{Theorem}

\begin{proof}
Let $J=B_1B_1^*$. By Lemma~\ref{L:Ruy} and Theorem 7.1 of~\cite{Ex}, we have the
following exact sequence of $K$-groups
\[
\begin{array}{ccccc}
K_0(J)&\overset{i_*-\theta_*^{-1}}{\longrightarrow}&K_0(A)&\overset{i_*}{\longrightarrow}&K_0(A_{\theta,\gamma})\\
\Big\uparrow&&&&\Big\downarrow\\
K_1(A_{\theta,\gamma})&\overset{i_*}{\longleftarrow}&K_1(A)&\overset{i_*-\theta_*^{-1}}{\longleftarrow}&K_1(J)\\
\end{array}
\]
It is {easy to see}  that $K_0(J)=0$, {$K_1(J)\cong C(Y, {\mathbb Z})$},
$K_0(A)\cong K_1(A)\cong\mathbb{Z}$. Note that
\[
K_1(J)\overset{i_*-\theta_*^{-1}}{\longrightarrow}K_1(A)
\]
is the composition of maps
\[
K_1(J)\overset{id}{\longrightarrow}K_1(A)
\]
and
\[
K_1(A)\overset{id-\theta_*^{-1}}{\longrightarrow}K_1(A).
\]
Since
\[
K_1(A)\overset{id-\theta_*^{-1}}{\longrightarrow}K_1(A)
\] is zero map, \[
K_1(J)\overset{i_*-\theta_*^{-1}}{\longrightarrow}K_1(A)
\]
is zero map. This gives the short exact sequence (\ref{TKgroup-2}).

To see it splits,
note that $Y$ may be identified with a compact subset of the unit line segment which in turn is viewed as a compact subset of the plane. Note also
that $K_1(C(Y))=\{0\}.$ It follows from the BDF-theory~\cite{BDF}  that
$Ext(C(Y))=\{0\}.$ Let $E$ be a unital essential extension
of the form:
$$
0\to {\cal K}\to E\to C(Y)\to 0.
$$
The fact that $Ext(C(Y))=\{0\}$
implies, in particular,
the short exact sequence
$$
0\to K_0({\cal K})\to K_0(E)\to K_0(C(Y))\to 0
$$
splits for any such $E,$ or,
$$
0\to {\mathbb Z}\to K_0(E)\to C(Y,{\mathbb Z})\to 0
$$
splits for any such group $K_0(E).$ It follows (from Brown's UCT) that
$Ext_{\mathbb Z}({\mathbb Z}, C(Y, {\mathbb Z}))=\{0\}.$
Therefore the short exact sequence (\ref{TKgroup-2}) splits.
\end{proof}

\begin{Corollary}
$K_i(A_{\theta, \gamma})$ is torsion free, $i=0,1.$  If $\gamma$ has finitely many zeros,
then $K_i(A_{\theta, \gamma})$ is free, $i=0,1.$
\end{Corollary}

\section{Classification of simple $C^*$-algebras of $A_{\theta, \gamma}$}

{
In this section, we will discuss the structure of $A_{\theta, \gamma}$ when
it is simple. } For recursive subhomogeneous algebras see~\cite{Phi}, Section 1.
Recall also that the Jiang-Su algebra $\cal{Z}$ is a unital simple $C^*$-algebra of recursive subhomogeneous $C^*$-algebra with one dimensional base spaces with a unique tracial state and with $K_0(\cal{Z})=\mathbb{Z}$ and
$K_1({\cal Z})=\{0\}$ (see~\cite{JS}).

\begin{Lemma}\label{NLzstable}
{
Let $\theta$ be an irrational number.
 Suppose that
$\gamma$ has at least one zero.
Then $A_{\theta, \gamma}$ is an inductive limit of
 recursive subhomogeneous $C^*$-algebras
with one dimensional base spaces. In particular, if $A_{\theta,\gamma}$ is simple, then $A_{\theta,\gamma}$ is
${\cal Z}$-stable, where ${\cal Z}$ is the Jiang-Su algebra~\cite{JS}.}
\end{Lemma}

\begin{proof}
{Let $Y$ be the set of zeros of $\gamma.$ It is a closed subset of ${\mathbb T}.$
Let $1/4>\ep_n>0$ be such that $\lim_{n\to\infty}\ep_n=0.$ Define
$$
Y_n=\{ x\in {\mathbb T}: {\rm dist}(x, Y)\le \ep_n\},\,\,\,n=1,2,....
$$
Define
\beq\label{NLzstable-1}
A_Y=C^*(C({\mathbb T}), uC_0({\mathbb T}\setminus Y))\andeqn\\
A_{Y_n}=C^*(C({\mathbb T}), uC_0({\mathbb T}\setminus Y_n)).
\eneq
}
{By  Theorem 2.3 of \cite{LP} (or Example 1.6 of~\cite{Phi}), $A_{Y_n}$ is a  recursive
subhomogeneous \CA\, with one dimensional base spaces.  Since $A_{Y}=\lim_{n\to\infty}A_{Y_n}$ (with inclusion maps), the first part of the lemma follows.}

{To see the second part, it follows from Theorem 1.6 of \cite{WinterPLMS04}
that each $A_{Y_n}$ has decomposition rank at most one. Therefore $A_Y$ has
decomposition rank one. Since we assume that $A_Y$ is simple, by Theorem 5.1 of \cite{Winterinv11}, $A_Y$ is ${\cal Z}$-stable. Note that
$A_Y=A_{\theta,\gamma}.$
}
\end{proof}

{
\begin{Lemma}\label{NTKtheory}
Suppose that $A_{\theta, \gamma}$ is simple and $Y$ is the set of zeros of $\gamma$. Let $\imath$ be the embedding of $A_{\theta, \gamma}=C^*(C({\mathbb T}), uC_0({\mathbb T}\setminus Y))\subset A_{\theta}$, and let $\rho_{A_{\theta, \gamma}}$ be the induced map of $K_0(A_{\theta, \gamma})$ into $K_0(A_{\theta})$. Then
\beq\label{NTKtheory-1}
\rho_{A_{\theta, \gamma}}(K_0(A_{\theta, \gamma}))={\mathbb Z}+{\mathbb Z}{\theta}
\andeqn {\rm ker}\rho_{A_{\theta, \gamma}}\cong C(Y,{\mathbb Z})/{\mathbb Z},
\eneq
where  ${\mathbb Z}$ is identified
with constant functions in $C(Y, {\mathbb Z}).$
Thus one has the following splitting short exact sequence:
\beq\label{NTKtheory-2}
0\to C(Y, {\mathbb Z})/{\mathbb Z}\to K_0(A_{\theta, \gamma})
\overset{\rho_{A_{\theta,\gamma}}}{\longrightarrow} {\mathbb Z}+{\mathbb Z}{\theta}\to 0.
\eneq
Moreover, in this case,
\beq\label{NTKtheory-2}
K_0(A_{\theta, \gamma})_+=\{0\}\cup\{x\in K_0(A_{\theta, \gamma}):
\rho_{A_{\theta,\gamma}}(x)>0\}
\eneq
and $K_0(A_{\theta, \gamma})$ is weakly unperforated and
has the Riesz interpolation property.
\end{Lemma}
}

\begin{proof}
{
Denote by $\phi: {\mathbb T}\to {\mathbb T}$ the rotation of the unit circle by $\theta$, i.e., $\phi(z)=e^{2\pi i\theta}z$ for $z\in \mathbb{T}$.
By the assumption of the lemma and Theorem~\ref{SimpleAtheta}, $\phi^n(Y)\cap Y=\emptyset$ for all integers $n\not=0.$
By Theorem 2.4 and Example 2.6 of \cite{Putnam}, one obtains the following
six term exact sequence:}
{
\[
\begin{array}{ccccc}
K_0(C(Y))&{\longrightarrow}&K_0(A_{\theta, \gamma})&\overset{\imath_{*0}}{\longrightarrow}&K_0(A_{\theta})\\
\Big\uparrow&&&&\Big\downarrow\\
K_1(A_{\theta})&\overset{\imath_{*1}}{\longleftarrow}&K_1(A_{\theta,\gamma})
&{\longleftarrow}&K_1(C(Y))\\
\end{array}
\]
Note that $Y$ is a proper closed subset of the circle. It follows
that $K_1(C(Y))=\{0\}$ and  $\imath_{*0}=\rho_{A_{\theta, \gamma}}$ is surjective.
Since
$K_0(A_{\theta})={\mathbb Z}+{\mathbb Z}\theta$ as an ordered subgroup of
${\mathbb R},$  ${\rm Ran}\rho_{A_{\theta,\gamma}}=\mathbb{Z}+\mathbb{Z}\theta$.   By Theorem \ref{T:K-groups}, $K_1(A_{\theta, \gamma})={\mathbb Z}.$ One then computes that
$$
{\rm ker}{\rho_{A_{\theta, \gamma}}}\cong C(Y, {\mathbb Z})/{\mathbb Z}.
$$
}
{
It is proved in Lemma~\ref{NLzstable} that $A_{\theta, \gamma}$ is ${\cal Z}$-stable. In particular, $K_0(A_{\theta, \gamma})$ has the strict comparison.
Therefore
\beq\label{NTKtheory-4}
K_0(A_{\theta, \gamma})_+=\{0\}\cup\{x\in K_0(A_{\theta, \gamma}): \rho_{A_{\theta, \gamma}}(x)>0\}.
\eneq
It follows that $K_0(A_{\theta, \gamma})$ is weakly unperforated and
has the Riesz interpolation property.
}
\end{proof}

For the convenience of the reader, we recall the meaning of tracial rank zero (or tracial topological rank zero) for simple $C^*$-algebras.

\begin{Definition}
\emph{
Let $A$ be a simple unital $C^*$-algebra. Then $A$ has tracial rank zero if for every subset ${\cal F}\subset A$, every $\epsilon>0$, and every nonzero positive element $c\in A$, there exists a projection $p\in A$ and a unital finite dimensional subalgebra $E\subset pAp$ such that:
\begin{enumerate}
\item[(1)] $\|[a,p]\|<\epsilon$ for all $a\in {\cal F}$.
\item[(2)] ${\rm dist}(pap, E)<\epsilon$ for all $a\in {\cal F}$.
\item[(3)] $1-p$ is Murray-von Neumann equivalent to a projection in $\overline{cAc}$.
\end{enumerate}
}
\end{Definition}
This definition is equivalent to the original one following from~\cite{Lin2}, Proposition 3.8.
\begin{Theorem}\label{T1}
{
Let $A_{\theta, \gamma}$ be a unital simple $C^*$-algebra. Then
$A_{\theta, \gamma}$ is a unital simple $A{\mathbb T}$-algebra of real rank zero. In particular,
$A_{\theta, \gamma}$ has tracial rank zero.}
\end{Theorem}

\begin{proof}
{
By Lemma~\ref{NLzstable} and Theorem~\ref{NTKtheory}, $A_{\theta, \gamma}$ is ${\cal Z}$-stable,
$K_0(A_{\theta, \gamma})$ is weakly unperforated and
has the Riesz interpolation property. Since $A_{\theta, \gamma}$ is an inductive limit of type I \CA s, it satisfies the universal coefficient theorem.
By Corollary~\ref{T:K-groups}, $K_i(A_{\theta, \gamma})$ is torsion free.
Therefore, by \cite{Ellinter}, there is a unital simple $A{\mathbb T}$-algebra $C$ of real rank zero such that
\beq\label{T1-1}
(K_0(C), K_0(C)_+,[1_C], K_1(C))\cong (K_0(A_{\theta, \gamma}), K_0(A_{\theta, \gamma}), [1_{A_{\theta, \gamma}}], K_1(A_{\theta, \gamma})).
\eneq
}

{
Let $U$ be a UHF-algebra of infinite type.
Consider $B=A_{\theta, \gamma}\otimes U.$ $B$ has a unique tracial state and is approximately divisible. Therefore its projections
separate the tracial state space. It follows from \cite{BKR} that $B$ has real rank zero. Since $B$ is ${\cal Z}$-stable, $B$ has strict comparison for projections.
Therefore $K_0(B)$ is weakly unperforated.
It follows from Lemma \ref{NLzstable} that $B$ is a locally type
I $C^*$-algebra.  Then,  by applying 5.16 of \cite{Lincrell}, $B$ has tracial rank zero. We also note that since $A_{\theta, \gamma}$ satisfies the universal coefficient theorem, so does $B.$}

{It follows from the classification theorem of \cite{LNadv} (Theorem 5.4) that
$C\otimes {\cal Z}\cong A_{\theta, \gamma}\otimes {\cal Z}.$
However, $C$ is ${\cal Z}$-stable and, by Lemma \ref{NLzstable}, $A_{\theta, \gamma}$ is also ${\cal Z}$-stable, one actually has}
\beq\label{T1-4}
C\cong A_{\theta, \gamma}.
\eneq

\end{proof}

\begin{Corollary}\label{T1c3}
{Let $\theta$ be an irrational number, $\gamma\in C({\mathbb T})$ be a non-negative function, let $Y$ be the set of zeros of $\gamma$ and let $\phi: {\mathbb T}\to {\mathbb T}$ be the homeomorphism by rotation of $\theta.$
Then the following are equivalent:}

{(1) $A_{\theta, \gamma}$ is simple;}

{(2) $A_{\theta, \gamma}$ has a unique tracial state;}

{(3) $\phi^n(Y)\cap Y=\emptyset$ for all integers $n\not=0;$ }

{(4) For each $\xi\in {\mathbb T},$ $Orb(\xi)\cap Y$ contains at most one point; }

{(5) $A_{\theta, \gamma}$ is a unital simple $A{\mathbb T}$-algebra of real rank zero.}
\end{Corollary}

\begin{Theorem}\label{T1c1}
{Let $\theta_1$ and $\theta_2$ be two irrational numbers,
$\gamma_1$ and $\gamma_2\in C({\mathbb T})$ be non-negative functions and let $Y_i$ be the set of zeros of $\gamma_i,$ $i=1,2.$
Suppose that $A_{\theta_i, \gamma_i}$ is simple,
or one of the equivalent conditions in Corollary~\ref{T1c3} satisfies.
Then  $A_{\theta_1, \gamma_1}\cong A_{\theta_2, \gamma_2}$  if and only if
the following hold:
\beq\label{T1c1-1}
\theta_1=\pm \theta_2 mod(\mathbb Z)\andeqn C(Y_1, {\mathbb Z})/{\mathbb Z}\cong C(Y_2, {\mathbb Z})/{\mathbb Z}.
\eneq
In particular, when  $\gamma_1$ has only finitely many zeros, then
$A_{\theta_1,\gamma_1}\cong A_{\theta_2, \gamma_2}$  if and only if $\theta_1=\pm \theta_2 mod{\mathbb{Z}}$ and
$\gamma_2$ has the same number of zeros.}
\end{Theorem}

\begin{proof}
{
We will prove the ``if" part only.
Note that we have $K_1(A_{\theta_1, \gamma_1})\cong K_1(A_{\theta_2, \gamma_2}).$ We may write, by Lemma \ref{NTKtheory},  that
\beq\label{T1c1-2}
K_0(A_{\theta_i, \gamma_i})=C(Y_i, {\mathbb Z})/{\mathbb Z}\oplus
(\mathbb Z}+{\mathbb Z\theta).
\eneq
}
{It follows that $K_0(A_{\theta_1, \gamma_1})\cong K_0(A_{\theta_2, \gamma_2}).$ In fact they are order isomorphic. By Theorem \ref{T1} both \CA s are
unital simple $A{\mathbb T}$-algebras of real rank zero. By the classification theorem they are isomorphic. }

\end{proof}

\begin{Corollary}\label{T1c2}
{
 With the same assumption as in \ref{T1c1}, if $Y_1$ and $Y_2$ are homeomorphic and $\theta_1=\pm \theta_2 \text{mod}\,{\mathbb Z},$ then
$A_{\theta_1, \gamma_1}\cong A_{\theta_2, \gamma_2}.$}
\end{Corollary}

\begin{Theorem}\label{Morita}
Let $\theta_1, \theta_2\in (0,1)$ be two irrational numbers, $\gamma_1, \gamma_2\in C({\mathbb T})$ be non-negative functions and let $Y_i$ be the set of zeros
of $\gamma_i,$ $i=0,1.$  Suppose that $A_{\theta_i, \gamma_i}$ is simple,
or one of the equivalent conditions in Corollary~\ref{T1c3} satisfies. Then
$A_{\theta_1,\gamma_1}$ and $A_{\theta_2, \gamma_2}$ are Morita equivalent if and only if ${\mathbb Z}+{\mathbb Z}\theta_1$ and
${\mathbb Z}+{\mathbb Z}\theta_2$ are order isomorphic and
\beq\label{Morita-1}
C(Y_1, {\mathbb Z})/{\mathbb Z}\cong C(Y_2, {\mathbb Z})/{\mathbb Z}.
\eneq

In particular, assuming, in addition,  $Y_1$ and $Y_2$ are both finite subsets, then $A_{\theta_1, \gamma_1}$ and $A_{\theta_2, \gamma_2}$ are Morita equivalent if and only if
${\mathbb Z}+{\mathbb Z}\theta_1$ and
${\mathbb Z}+{\mathbb Z}\theta_2$ are order isomorphic
and $Y_1$ and $Y_2$ have the same number of points.

\end{Theorem}

\begin{proof}
Suppose that $h_1: {\mathbb Z}+{\mathbb Z}\theta_1\to
{\mathbb Z}+{\mathbb Z}\theta_2$ is an order isomorphism and
$h_2: C(Y_1, {\mathbb Z})/{\mathbb Z}\to C(Y_2, {\mathbb Z})/{\mathbb Z}$ is an isomorphism as groups.
There is an injective homomorphism $\imath_i: {\mathbb Z}+{\mathbb Z}\theta_i\to
K_0(A_{\theta_i, \gamma_i})$ such that
$$
\rho_{A_{\theta_i, \gamma_i}}\circ \imath_i={\rm id}_{{\mathbb Z}+{\mathbb Z}\theta_i},\,\,\, i=1,2.
$$
We write
$$
K_0(A_{\theta_i, \gamma_i})=C(Y_i, {\mathbb Z})/{\mathbb Z}\oplus
\imath_i({\mathbb Z}+{\mathbb Z}\theta_i),
$$
$i=1,2.$

Define $h_3: K_0(A_{\theta_1, \gamma_1})\to K_0(A_{\theta_2, \gamma_2})$ by \beq\label{Morita-3}
h_3|_{{\rm ker}\rho_{A_{\theta_1, \gamma_1}}}=h_2
\eneq
and
\beq\label{Morita-4}
h_3(x)=\imath_2\circ h_1\circ \rho_{A_{\theta_1, \gamma_1}}(x).
\eneq
for $x\in \imath_1({\mathbb Z}+{\mathbb Z}\theta_1).$
It is easy to verify that $h_3$ is an order isomorphism from $K_0(A_{\theta_1, \gamma_1})$ onto $K_0(A_{\theta_2, \gamma_2}).$
We also have $K_1(A_{\theta_1, \gamma_1})={\mathbb Z}=K_1(A_{\theta_2, \gamma_2}).$
Since both $A_{\theta_1, \gamma_1}$ and $A_{\theta_2, \gamma_2}$ are
unital simple $A{\mathbb T}$-algebras of real rank zero, by the classification results mentioned earlier, $A_{\theta_1, \gamma_1}$ and $A_{\theta_2, \gamma_2}$ are stably isomorphic. In other words,
$A_{\theta_1, \gamma_1}$ and $A_{\theta_2, \gamma_2}$ are Morita equivalent.

Conversely,
if $A_{\theta_1, \gamma_1}\otimes {\cal K}\cong A_{\theta_2, \gamma_2}\otimes {\cal K},$ then $K_0(A_{\theta_1,\gamma_1})$ and $K_0(A_{\theta_2, \gamma_2})$ are order isomorphic. Denote by
$h_0$ the order isomorphism.
This implies, in particular,  $h_0$ maps ${\rm ker}\rho_{A_{\theta_1, \gamma_1}}$ isomorphically onto ${\rm ker}\rho_{A_{\theta_2,\gamma_2}}$ which implies that
$$
C(Y_1, {\mathbb Z})/{\mathbb Z}={\rm ker}\rho_{A_{\theta_1, \gamma_1}}\cong {\rm ker}\rho_{A_{\theta_2,\gamma_2}}=C(Y_2, {\mathbb Z})/{\mathbb Z}.
$$
Therefore $h_0$ induces an order isomorphism from $\rho_{A_{\theta_1, \gamma_1}}(K_0(A_{\theta_1,\gamma_1}))$ onto $\rho_{A_{\theta_2, \gamma_2}}(K_0(A_{\theta_2, \gamma_2}))$ which implies that
${\mathbb Z}+{\mathbb Z}\theta_1$ and ${\mathbb Z}+{\mathbb Z}\theta_2$ are order isomorphic.
\end{proof}

Let ${\rm GL}(2,\mathbb{Z})$ denote the group of $2\times 2$ matrices with entries in $\mathbb{Z}$ and with determinant $\pm 1$, and let ${\rm GL}(2,\mathbb{Z})$ act on the set of irrational numbers by
\[
\begin{pmatrix}
a&b\\
c&d
\end{pmatrix}\alpha=\frac{a\alpha+b}{c\alpha+d}.
\]
By Corollary 2.6 of~\cite{Ri} (or Lemma 4.7 of~\cite{Sh}), $\mathbb{Z}+\mathbb{Z}\theta_1$ and $\mathbb{Z}+\mathbb{Z}\theta_2$ are ordered isomorphic if and only if $\theta_1$ and $\theta_2$ are in the same orbit of ${\rm GL}(2,\mathbb{Z})$. Thus we obtain the following corollary.
\begin{Corollary}\label{Riffelcond}
Let $\theta_1, \theta_2\in (0,1)$ be two irrational numbers, $\gamma\in C({\mathbb T})$ be non-negative functions and let $Y$ be the set of zeros of $\gamma$. Suppose that $A_{\theta_i, \gamma}$ is simple,
or one of the equivalent conditions in Corollary~\ref{T1c3} satisfies.  Then $A_{\theta_1,\gamma}$ and  $A_{\theta_2,\gamma}$ are morita equivalent if and only if $\theta_1$ and $\theta_2$ are in the same orbit under the action of ${\rm GL}(2,\mathbb{Z})$ on  irrational numbers.
\end{Corollary}

\section{The $C^*$-algebra generated by $u+\lambda v$}
\begin{Proposition}\label{P:algebras generated by g(w)}
Let $R$ be the hyperfinite type ${\rm II}_1$ factor with two unitary
generators $u,v$ such that $vu=e^{2\pi i\theta}uv$.
 If $f(z)\in C(\mathbb{T})$ and $m(\{z|f(z)=0\})=0$, then the von Neumann
subalgebra generated by $uf(v)$ and $v$ is $R$. Furthermore,
$C^*(uf(v), v)=C^*(u,v)$ if and only if $f(z)\neq 0$ for all $z\in
\mathbb{T}$.
\end{Proposition}
\begin{proof}
Let $M$ be the von Neumann algebra generated by $uf(v)$ and $v$.
Since $m(\{z|f(z)=0\})=0$, $f(v)^{-1}$ is affiliated with $M$, i.e., the spectral projections of the unbounded operator $f(v)^{-1}$ are in $M$. Hence $u=uf(v)\cdot f(v)^{-1}$ is affiliated with $M$. Since $u$ is a
bounded operator, $u\in M$ and therefore $R\subseteq M$ and $M=R$.

If $f(z)\neq 0$ for all $z\in \mathbb{T}$, then $f(v)$ is an
invertible operator in $C^*(v)$. Hence $u=uf(v)\cdot f(v)^{-1}$ is
in the $C^*$-subalgebra generated by $uf(v)$ and $v$. Therefore,
$C^*(uf(v), v)=C^*(u,v)$. Conversely, suppose $f(z_0)=0$ for some
$z_0\in \mathbb{T}$. By Theorem~\ref{T:K-groups}, $K_1(C^*(uf(v), v))\cong\mathbb{Z}$.
Therefore, $C^*(uf(v), v)\neq C^*(u,v)$.
\end{proof}

\begin{Theorem}\label{T:generators}
Let $R$ be the hyperfinite type ${\rm II}_1$ factor with  two unitary
generators $u,v$ such that $vu=e^{2\pi i\theta}uv$.
Then the von Neumann subalgebra generated by $u+\lambda v$ is $R$
for $\lambda>0$. Furthermore, $C^*(u+\lambda v)=C^*(u,v)$ if
$\lambda\neq 1$ while $C^*(u+v)$ is  a proper simple
$C^*$-subalgebra of $C^*(u,v)$ which has a unique trace,
$K_1(C^*(u+v))\cong \mathbb{Z}$, and there is an order isomorphism
of $K_0(C^*(u+v))$ onto $\mathbb{Z}+\mathbb{Z}\theta$. {Moreover,
$C^*(u+v)$ is a unital simple $A{\mathbb T}$-algebra of tracial rank zero.}
\end{Theorem}
\begin{proof}
 Note that
\[
(u+\lambda v)(u+\lambda v)^*=(u+\lambda v)(u^*+\lambda
v^*)=\lambda e^{-2\pi i\theta}u^*v+\lambda uv^*+1+\lambda^2
\]
and
\[
(u+\lambda v)^*(u+\lambda v)=(u^*+\lambda v^*)(u+\lambda
v)=u^*v+e^{-2\pi i\theta}\lambda uv^*+1+\lambda^2.
\]
Hence $u^*v,uv^*\in C^*(u+\lambda v)$. Let $w=u^*v$. Thus
$C^*(u+\lambda v)=C^*(u+\lambda v, w)=C^*(u(1+\lambda w), w)$. By
Proposition~\ref{P:algebras generated by g(w)}, the von Neumann
subalgebra generated by $u+\lambda v$ is $R$ for $\lambda>0$, and
$C^*(u+\lambda v)=C^*(u,v)$ if $\lambda\neq 1$ while $C^*(u+v)$ is
a proper $C^*$-subalgebra of $C^*(u,v)$. Note that $u+v$ and $w$
satisfy (\ref{U1})-(\ref{U4}) for $\theta$ and
$\gamma(z)=|1+z|^2$. By Proposition~\ref{P:U unique trace},
Theorem~\ref{T:simple}, Theorem~\ref{T:Rieffel projection}, and
Theorem~\ref{T:K-groups}, $C^*(u+v)$ is a simple
algebra with a unique trace, $K_1(C^*(u+v))\cong \mathbb{Z}$, and
there is an order isomorphism of $K_0(C^*(u+v))$ onto
$\mathbb{Z}+\mathbb{Z}\theta$. By Theorem~\ref{T1}, $C^*(u+v)$ is a unital simple $A{\mathbb T}$-algebra of tracial rank zero.
\end{proof}

\begin{Corollary}
$C^*(u+v)$ is not $\ast$-isomorphic to $C^*(u,v)$.
\end{Corollary}

\section{Spectral radius of $u+\lambda v$}
In this section, we assume that $0\leq \lambda \leq 1$. Let $\alpha=e^{2\pi i \theta}$ and $w=u^*v$. Then $w$ is a Haar unitary operator in $R$, i.e., $\tau(w^n)=\tau((w^*)^n)=0$ for all $n\in \mathbb{N}$.  Note that
\[
u+\lambda v=u(1+\lambda u^*v)=u(1+\lambda w),
\]
\[
(u+\lambda v)^2=(u+\lambda v)u(1+\lambda w)=(u^2+\alpha\lambda  uv)(1+\lambda w)=u^2(1+\alpha \lambda  u^*v)(1+\lambda w)=u^2(1+\alpha \lambda w)(1+\lambda w),
\]
\[
(u+\lambda v)^3=(u+\lambda v)u^2(1+\alpha\lambda w)(1+\lambda w)=(u^3+\alpha\lambda u^2v)(1+\alpha\lambda w)(1+\lambda w)=u^3(1+\alpha^2 \lambda w)(1+\alpha\lambda w)(1+\lambda w).
\]
By induction, we have
\begin{equation}\label{E:(u+v)^{-1}}
(u+\lambda v)^n=u^n(1+\lambda w)(1+\alpha\lambda w)\cdots(1+\alpha^{(n-1)}\lambda w),\quad \forall n\in\mathbb{N}
\end{equation}
Let $r(u+\lambda v)$ be the spectral radius of $u+\lambda v$. Then
\[
r(u+\lambda v)=\lim_{n\rightarrow+\infty} \|(u+\lambda v)^n\|^{1/n}=\lim_{n\rightarrow+\infty} \|(1+\lambda w)(1+\alpha \lambda w)\cdots(1+\alpha^{(n-1)}\lambda w)\|^{1/n}.
\]
Since $w=u^*v$ is a Haar unitary operator, we may identify $w$
with the multiplication operator $M_z$ on $L^2(\mathbb{T},m)$,
where $m$ is the Haar measure on $\mathbb{T}$. Hence,
\begin{eqnarray}
\nonumber \|(u+\lambda v)^n\|^{1/n}&=&\|(1+\lambda w)(1+\alpha\lambda  w)\cdots(1+\alpha^{(n-1)}\lambda w)\|^{1/n}\\
\nonumber &=&\|(1+\lambda M_z)(1+\alpha\lambda  M_z)\cdots(1+\alpha^{(n-1)}\lambda M_z)\|^{1/n}\\
\nonumber&=&\left(\max_{z\in \mathbb{T}}\left|(1+\lambda
z)(1+\alpha \lambda z)\cdots (1+\alpha^{(n-1)}\lambda
z)\right|\right)^{1/n}.
\end{eqnarray}
Let $z=e^{i2\pi x}$, $x\in [0,1]$. Then simple calculation shows that
\[
\left|(1+\lambda z)(1+\alpha\lambda z)\cdots (1+\alpha^{(n-1)}\lambda z)\right|=\left(\prod_{k=0}^{n-1}\left(1+\lambda^2+2\lambda\cos(2\pi(x+k\theta))\right)\right)^{\frac{1}{2}}.
\]
So
\begin{equation}\label{E:(u+v)^n}
\|(u+\lambda v)^n\|^{1/n}=\max_{x\in[0,1]}\left(\prod_{k=0}^{n-1}\left(1+\lambda^2+2\lambda\cos(2\pi(x+k\theta))\right)\right)^{\frac{1}{2n}}.
\end{equation}

\begin{Lemma}\label{L:integral}
For $0\leq \lambda \leq 1$,
\[
\int_0^1\ln(1+\lambda^2+2\lambda \cos 2\pi x)dx=0.
\]
\end{Lemma}
\begin{proof}
For $0\leq \lambda\leq 1$, let
\[
f(\lambda)=\int_0^1\ln(1+\lambda^2+2\lambda \cos 2\pi x)dx.
\]
Then $f(\lambda)$ is continuous on $[0,1]$, differentiable in $(0,1)$, and $f(0)=0$. Note that for $0<\lambda<1$,
\begin{eqnarray*}
  f'(\lambda) &=& \int_0^1\frac{2\lambda+2\cos 2\pi x}{1+\lambda^2+2\lambda\cos 2\pi x}dx \\
   &=& \frac{1}{2\pi i}\int_{\mathbb{T}}\frac{2\lambda+z+\frac{1}{z}}{1+\lambda^2+\lambda\left(z+\frac{1}{z}\right)}\cdot\frac{dz}{z} \\
   &=& \frac{1}{2\pi i}\int_{\mathbb{T}}\frac{2\lambda z+z^2+1}{(1+\lambda^2)z+\lambda z^2+\lambda}\cdot\frac{dz}{z} \\
  &=& \frac{1}{2\pi i}\int_{\mathbb{T}}\frac{2\lambda z+z^2+1}{(\lambda z+1)(z+\lambda)z}dz \\
   &=& \text{Res}\left(\frac{2\lambda z+z^2+1}{(\lambda z+1)(z+\lambda)z};0 \right)+\text{Res}\left(\frac{2\lambda z+z^2+1}{(\lambda z+1)(z+\lambda)z};-\lambda\right)\\
   &=& \frac{1}{\lambda}-\frac{1}{\lambda}=0.
\end{eqnarray*}
So for $0\leq \lambda \leq 1$, $f(\lambda)=0$.

\end{proof}

\begin{Lemma}\label{L:Ergodic}
Let $0<\lambda\leq 1$. Then for almost all $x\in [0,1]$,
\[
\lim_{n\rightarrow +\infty}\left(\prod_{k=0}^{n-1}\left(1+\lambda^2+2\lambda\cos(2\pi(x+k\theta))\right)\right)^{\frac{1}{2n}}=1.
\]
\end{Lemma}
\begin{proof}
We only need to show that for almost all $x\in [0,1]$,
\[
\lim_{n\rightarrow+\infty}\frac{1}{n}\sum_{k=0}^{n-1}
\ln \left(1+\lambda^2+2\lambda\cos(2\pi(x+k\theta))\right)=0.
\]
Let $f(x)=\ln (1+\lambda^2+2\lambda \cos 2\pi x)$. If $0<\lambda<1$, then
\[
2\ln(1-\lambda)\leq f(x)\leq 2\ln (1+\lambda),\quad \forall x\in [0,1].
\]
 So $f(x)\in L^1[0,1]$. If $\lambda=1$, then
\[
f(x)=\ln(2+2\cos 2\pi x)=2\ln 2+2\ln|\cos \pi x|
\]
and so
\[
|f(x)|\leq 2\ln 2-2\ln|\cos \pi x|,\quad \forall x\in [0,1].
\]
By Lemma~\ref{L:integral}, $\int_0^1 f(x)dx=0$, which implies that $\int_0^1 2\ln|\cos \pi x|dx=-2\ln 2$.
Therefore, $\int_0^1 |f(x)|dx\leq 4\ln 2$ and $f(x)\in L^1[0,1]$.

Let $T:x\rightarrow x+\theta (\text{mod}\, 1)$. Then $T$ is a measure preserving ergodic transformation of $[0,1]$. By  Birkhoff's Ergodic theorem and Lemma~\ref{L:integral}, for almost all $x\in [0,1]$
\[
\lim_{n\rightarrow\infty}\frac{1}{n}\sum_{k=0}^{n-1}\ln \left(1+\lambda^2+2\lambda\cos(2\pi(x+k\theta))\right)=\int_0^1\ln(1+\lambda^2+2\lambda \cos 2\pi x)dx=0.
\]
\end{proof}

\begin{Corollary}\label{C:r geq 1}
For $0<\lambda\leq 1$, $r(u+\lambda v)\geq 1$.
\end{Corollary}
\begin{proof}
Let $\epsilon>0$. By Lemma~\ref{L:Ergodic}, there is an $x\in [0,1]$ and $N\in \mathbb{N}$ such that for all $n\geq N$,
\[
 \left(\prod_{k=0}^{n-1}\left(1+\lambda^2+2\lambda\cos(2\pi(x+k\theta))\right)\right)^{\frac{1}{2n}}\geq 1-\epsilon.
\]
By equation (\ref{E:(u+v)^n}), for $n\geq N$,
\[
\|(u+\lambda v)^n\|^{1/n}\geq 1-\epsilon.
\]
This implies that
\[r(u+\lambda v)=\lim_{n\rightarrow+\infty}\|(u+\lambda v)^n\|^{1/n}\geq 1-\epsilon.\]
 Since $\epsilon>0$ is arbitrary, $r(u+\lambda v)\geq 1$.
\end{proof}

Let $\theta\in (0,1)$ be an irrational number and let $\alpha=e^{2\pi i\theta}$.

\begin{Lemma}\label{L:irrational rotation}
Given $\epsilon>0$ and $N\in \mathbb{N}$. Then there exists $N'\in \mathbb{N}$ such that for $n\geq N'$ and every arc $\Gamma$ of the unit circle $\mathbb{T}$ with length $\frac{2\pi}{N}$, there exits $\frac{n}{N}+r$ points of $1,\alpha,\cdots, \alpha^{n-1}$ in $\Gamma$ with $\left|\frac{r}{n}
\right|<\epsilon$.
\end{Lemma}
\begin{proof}
Since $\theta\in (0,1)$ is irrational, $\{\alpha^k:\,k\in\mathbb{N}\}$ is dense in the unit circle $\mathbb{T}$. Therefore, there exists $m\in \mathbb{N}$ such that for every $0\leq \varphi\leq 2\pi$, there exists $1\leq k\leq m$ such that $|(\varphi-2k\pi \theta){\rm mod} 2\pi |<\frac{\epsilon}{8}$. By Birkhoff's ergodic theorem, there exists an arc $\Gamma_1$ of the unit circle with length $l(\Gamma_1)=2\pi \left(\frac{1}{N}-\frac{\epsilon}{4} \right)$ and
\[
\lim_{n\rightarrow\infty}\frac{\chi_{\Gamma_1}(1)+\chi_{\Gamma_1}(\alpha)+\cdots+\chi_{\Gamma_1}(\alpha^{n-1})}{n}=\frac{l(\Gamma_1)}{2\pi}
=\frac{1}{N}-\frac{\epsilon}{4}.
\]
Let $N_1$ be sufficiently large such that $\frac{m}{N_1}<\frac{\epsilon}{2}$ and if $n\geq N_1$ then
\[
\frac{\chi_{\Gamma_1}(1)+\chi_{\Gamma_1}(\alpha)+\cdots+\chi_{\Gamma_1}(\alpha^{n-1})}{n}
\geq \frac{1}{N}-\frac{\epsilon}{4}-\frac{\epsilon}{4}=\frac{1}{N}-\frac{\epsilon}{2}.
\]
Let $e^{2\pi i\theta}$ and $e^{2\pi i (\theta+2\pi/N)}$ be the ending points of the arc $\Gamma$. Let $\Gamma_1'\subset \Gamma$ be the arc of $\mathbb{T}$ with ending points $e^{2\pi i\theta+(\pi/4)\epsilon i}$ and $e^{2\pi i (\theta+2\pi/N)-(\pi/4)\epsilon i}$. Then  there exists an $\varphi$ with $0\leq \varphi\leq 2\pi$ such that we can rotate $\Gamma_1$ by angle $\varphi$ to obtain $\Gamma_1'$. So if $\{\alpha^{k_1},\cdots,\alpha^{k_s}\}\subseteq \Gamma_1$ with $0\leq k_1<k_2<\cdots<k_s\leq n-1$, then $\{\alpha^{k_1}e^{i\varphi},\cdots, \alpha^{k_s}e^{i\varphi}\}\subseteq \Gamma_1'\subseteq \Gamma$. Since $|(\varphi-2k\pi \theta){\rm mod} 2\pi|<\frac{\epsilon}{8}$ for some $1\leq k\leq m$,
\[
\{\alpha^{k_1}e^{2k\pi \theta i},\cdots,\alpha^{k_{s-m}}e^{2k\pi \theta i}\}\subset \Gamma.
\]
Since $k_{s-m}\leq n-m$, $\{\alpha^{k_1+k},\cdots,\alpha^{k_{s-m}+k}\}\subset \Gamma$. So $\Gamma$ contains at least $n\left(\frac{1}{N}-\frac{\epsilon}{2} \right)-m=n\left(\frac{1}{N}-\epsilon \right)$ points of $1,\alpha,\cdots,\alpha^{n-1}$.

By Birkhoff's ergodic theorem, there exists an arc $\Gamma_2$ of the unit circle with length $l(\Gamma_2)=2\pi \left(\frac{1}{N}+\frac{\epsilon}{4} \right)$ and
\[
\lim_{n\rightarrow\infty}\frac{\chi_{\Gamma_2}(1)+\chi_{\Gamma_2}(\alpha)+\cdots+\chi_{\Gamma_2}(\alpha^{n-1})}{n}=\frac{l(\Gamma_2)}{2\pi}
=\frac{1}{N}+\frac{\epsilon}{4}.
\]
Let $N_2$ be sufficiently large such that $\frac{m}{N_2}<\frac{\epsilon}{2}$ and if $n\geq N_2$ then
\[
\frac{\chi_{\Gamma_2}(1)+\chi_{\Gamma_2}(\alpha)+\cdots+\chi_{\Gamma_2}(\alpha^{n-1})}{n}
\leq \frac{1}{N}+\frac{\epsilon}{4}+\frac{\epsilon}{4}=\frac{1}{N}+\frac{\epsilon}{2}.
\]
Let $e^{2\pi i\theta'}$ and $e^{2\pi i \theta'+2\pi(1/N+\epsilon/4)i}$ be the ending points of the arc $\Gamma_2$. Let $\Gamma_2'\subset \Gamma_2$ be the arc of $\mathbb{T}$ with ending points $e^{2\pi i\theta'+(\pi/4)\epsilon i}$ and $e^{2\pi i (\theta'+1/N)+(\pi/4)\epsilon i}$. Then  there exists an $\varphi'$ with $0\leq \varphi'\leq 2\pi$ such that we can rotate $\Gamma$ by angle $\varphi'$ to obtain $\Gamma_2'$. So if $\{\alpha^{j_1},\cdots,\alpha^{j_r}\}\subseteq \Gamma$ with $0\leq j_1<j_2<\cdots<j_s\leq n-1$, then $\{\alpha^{j_1}e^{i\varphi'},\cdots, \alpha^{j_r}e^{i\varphi'}\}\subseteq \Gamma_2'\subseteq \Gamma$. Since $|(\varphi'-2k'\pi \theta){\rm mod} 2\pi|<\frac{\epsilon}{8}$ for some $1\leq k'\leq m$,
\[
\{\alpha^{j_1}e^{2k'\pi \theta i},\cdots,\alpha^{j_{r-m}}e^{2k'\pi \theta i}\}\subset \Gamma_2.
\]
 So $\Gamma$ contains at most $n\left(\frac{1}{N}+\frac{\epsilon}{2} \right)+m=n\left(\frac{1}{N}+\epsilon \right)$ points of $1,\alpha,\cdots,\alpha^{n-1}$. Let $N'=\max\{N_1,N_2\}$. Then we obtain the lemma.
\end{proof}

Now we prove the main result of this section.
\begin{Lemma}\label{L:r=1}
For $0<\lambda\leq 1$, $r(u+\lambda v)=1$.
\end{Lemma}
\begin{proof}
By Corollary~\ref{C:r geq 1}, we need to prove that $r(u+\lambda v)\leq 1$.
Let $\epsilon>0$. Note that
\[
\lim_{n\rightarrow\infty}\frac{1}{2n}
\left(\sum_{k=0}^{n-1}\ln\left(1+\lambda^2+2\lambda\cos\left(\frac{k\pi}{n}\right)\right)+
\sum_{k=n+1}^{2n}\ln\left(1+\lambda^2+2\lambda\cos\left(\frac{k\pi}{n}\right)\right)\right)
\]
\[
=\int_0^1\ln(1+\lambda^2+2\lambda \cos 2\pi x)dx=0.
\]
There is $N\in \mathbb{N}$ such that
\[
\frac{1}{2N}
\left(\sum_{k=0}^{N-1}\ln\left(1+\lambda^2+2\lambda\cos\left(\frac{k\pi}{N}\right)\right)+
\sum_{k=N+1}^{2N}\ln\left(1+\lambda^2+2\lambda\cos\left(\frac{k\pi}{N}\right)\right)\right)<\epsilon.
\]
Let
\[
M(\lambda)=\max_{0\leq k\leq 2N, k\neq N}\left|\ln\left(1+\lambda^2+2\lambda\cos\left(\frac{k\pi}{N}\right)\right)\right|.
\]
Then for $0<\lambda\leq 1$, $M(\lambda)<\infty$. Divide the unit
circle $\mathbb{T}$ into $2N$ equal sections $A_1,\cdots, A_{2N}$.
By Lemma~\ref{L:irrational rotation}, there exists $N'$ such that for all
$n\geq N'$ and all $x\in [0,1]$, if $A_k$ contains
$n/(2N)+r_{k}(x)$ points of $e^{2\pi ix}, \alpha e^{2\pi ix},
\cdots, \alpha^{n-1} e^{2\pi ix}$, then $\displaystyle
\frac{\sum_{k=1}^{2N}|r_k(x)|}{n}<\frac{\epsilon}{M(\lambda)}. $
Note that $\cos 2\pi x$ is decreasing for $x\in [0,1/2]$ and
increasing for $x\in [1/2,1]$. Therefore, for all $x\in [0,1]$,

\[
\frac{1}{n}\sum_{k=0}^{n-1}
\ln \left(1+\lambda^2+2\lambda\cos(2\pi(x+k\theta))\right)\leq \frac{1}{n}\sum_{k=0}^{N-1}\left(\frac{n}{2N}+r_{k+1}(x)\right)\ln\left(1+\lambda^2+2\lambda\cos\left(\frac{k\pi}{N}\right)\right)
\]
\[
+\frac{1}{n}\sum_{k=N+1}^{2N}\left(\frac{n}{2N}+r_k(x)\right)\ln\left(1+\lambda^2+2\lambda\cos\left(\frac{k\pi}{N}\right)\right)
\]
\[
=\frac{1}{2N}
\left(\sum_{k=0}^{N-1}\ln\left(1+\lambda^2+2\lambda\cos\left(\frac{k\pi}{N}\right)\right)+
\sum_{k=N+1}^{2N}\ln\left(1+\lambda^2+2\lambda\cos\left(\frac{k\pi}{N}\right)\right)\right)
\]
\[
+\frac{1}{n}\sum_{k=0}^{N-1}r_{k+1}(x)\ln\left(1+\lambda^2+2\lambda\cos\left(\frac{k\pi}{N}\right)\right)
+\frac{1}{n}\sum_{k=N+1}^{2N}r_k(x)\ln\left(1+\lambda^2+2\lambda\cos\left(\frac{k\pi}{N}\right)\right)
\]
\[
<\epsilon+\frac{1}{n}\sum_{k=1}^{2N}
|r_k(x)|M(\lambda)<2\epsilon.
\]
This implies that for all $n\geq N'$ and $x\in [0,1]$,
\[
\left(\prod_{k=0}^{n-1}\left(1+\lambda^2+2\lambda\cos(2\pi(x+k\theta))\right)\right)^{\frac{1}{2n}}\leq e^{2\epsilon}.
\]
By equation(\ref{E:(u+v)^n}),
$\|(u+\lambda v)^n\|^{1/n}\leq e^{2\epsilon}$ for all $n\geq N'$. So $r(u+\lambda v)\leq e^{2\epsilon}$. Since $\epsilon>0$ is arbitrary, $r(u+\lambda v)\leq 1$.
\end{proof}

\section{Strongly irreducible operators relative to  type ${\rm II}_1$ factors}

An operator $T$ in a type ${\rm II}_1$ factor $M$ is
called \emph{irreducible} if $\{T,T^*\}'\cap M=\mathbb{C}1$, i.e.,
the von Neumann subalgebra generated by $T$ is an irreducible
subfactor of $M$.
\begin{Proposition}
 Every separable type ${\rm II}_1$
factor $M$ contains an irreducible operator.
\end{Proposition}
\begin{proof}
 By~\cite{Po}, every separable type ${\rm II}_1$
factor $M$ contains an irreducible hyperfinite factor. Since hyperfinite factor is generated by an operator $T$, it follows that $T$ is an irreducible operator in $M$.
\end{proof}

Recall that an operator $T$ in $B(H)$ is a \emph{strongly
irreducible operator} if  there is no nontrivial idempotents in
$\{T\}'$. Strongly irreducible operators are generalizations of
Jordan blocks in matrix algebras. A rich theory has been set up on
this class of operators in the past twenty years
(see~\cite{JW1,JW2}).
 Let $M$ be a type ${\rm II}_1$ factor. An operator $T\in M$ is
called a \emph{strongly irreducible operator relative to $M$} if
$\{T\}'\cap M=\mathbb{C} 1$. In this section we will give explicit examples of strongly irreducible operators in hyperfinite  ${\rm II}_1$ factors.

Let $A_\theta$ be the universal irrational rotation $C^*$-algebra with two
unitary generators $u,v$ such that $vu=e^{2\pi i\theta}uv$. Then there
exists a unique trace $\tau$ on $A_\theta$. Applying the GNS-construction to $\tau$, we may
assume that $A_\theta$ acts on $L^2(A_\theta,\tau)$. Let $R$ be the strong operator
closure of $A_\theta$. Then $R$ is the hyperfinite type ${\rm II}_1$ factor with a unique trace $\tau$.
Recall that $u,v$ in $R$ satisfy the following properties:
\begin{enumerate}
\item $\tau(u^n)=\tau(v^n)=0$ for all integers $n\neq 0$;

\item $vu=e^{2\pi i\theta}uv$;

\item $\{u^mv^n:\,m,n\in\mathbb{Z}\}$ is an orthonormal basis of
$L^2(R)=L^2(R,\tau)$, where $u^mv^n$ is viewed as an element of
$L^2(R)$.

\end{enumerate}

The following theorem is the main result of this section.

\begin{Theorem}\label{L:strongly irreducible}
For every irrational number $\theta\in (0,1)$, $u+v$ is a strongly
irreducible operator relative $R$, i.e., there exists no
nontrivial idempotents in $\{u+v\}'\cap R$.
\end{Theorem}
\begin{proof}
Let  $x\in \{u+v\}'\cap R$. By condition 3 above
Theorem~\ref{L:strongly irreducible}, $x=\sum_{m,n\in
\mathbb{Z}}\alpha_{m,n}u^mv^n$ and $\sum_{m,n\in
\mathbb{Z}}|\alpha_{m,n}|^2=\tau(x^*x)<\infty$. By condition 2
above Theorem~\ref{L:strongly irreducible},
\begin{equation}\label{E:(u+v)x}
(u+v)x=(u+v)\sum_{m,n\in
\mathbb{Z}}\alpha_{m,n}u^mv^n=\sum_{m,n\in
\mathbb{Z}}\alpha_{m,n}u^{m+1}v^n+\sum_{m,n\in
\mathbb{Z}}\alpha_{m,n}e^{2\pi im\theta}u^mv^{n+1}
\end{equation}
and
\begin{equation}\label{E:x(u+v)}
x(u+v)=\sum_{m,n\in
\mathbb{Z}}\alpha_{m,n}u^mv^n(u+v)=\sum_{m,n\in
\mathbb{Z}}\alpha_{m,n}e^{2\pi in\theta}u^{m+1}v^n+\sum_{m,n\in
\mathbb{Z}}\alpha_{m,n}u^mv^{n+1}.
\end{equation}
 By condition 3 above Theorem~\ref{L:strongly irreducible}, $\{u^mv^n:\,m,n\in\mathbb{Z}\}$ is an orthonormal basis of $L^2(R)$. Comparing the coefficients of the term $u^mv^n$ in (\ref{E:(u+v)x}) and (\ref{E:x(u+v)}), we have
\[
\alpha_{m-1,n}+\alpha_{m,n-1}e^{2\pi
im\theta}=\alpha_{m-1,n}e^{2\pi in\theta}+\alpha_{m,n-1},
\]
which is equivalent to
\begin{equation}\label{E:2.3}
\alpha_{m-1,n}(1-e^{2\pi in\theta})=\alpha_{m,n-1}(1-e^{2\pi
im\theta}).
\end{equation}

Since $\theta$ is an irrational number, $1-e^{2\pi ik\theta}\neq
0$ for $k\neq 0$. Let $n=0$ in equation~(\ref{E:2.3}). We have
$\alpha_{m,-1}=0$ for $m\neq 0$. Let $n=-1$ in equation
(\ref{E:2.3}). We have $\alpha_{m,-2}=0$ for $m\neq 0$, $m\neq 1$.
In general, let $n=-k$ in equation (\ref{E:2.3}). We have
$\alpha_{m,-k-1}=0$ for $m\neq 0$, $\cdots$, $m\neq k$. On the
other hand, let $m=0$ in equation (\ref{E:2.3}). We have
$\alpha_{-1,n}=0$ for $n\neq 0$. Similarly, in general we have
$\alpha_{-k-1,n}=0$ for $n\neq 0$, $\cdots$, $n\neq k$. So we have
$\alpha_{m,n}=0$ if either both $m<0$ and $n<0$ or $m=-n\neq 0$.

The motivation of the following part is to prove that
$\alpha_{m,n}=0$ if either $m<0$ or $n<0$. We only need to show
that $\alpha_{m,-k-m}=0$ and $\alpha_{-k-m,m}=0$ for $k\geq 1$ and
$m\geq 0$. Repeat using equation (\ref{E:2.3}), we have
\begin{equation}\label{E:2.4}
\alpha_{m,-k-m}=\alpha_{0,-k}\frac{1-e^{-2\pi ik\theta}}{1-e^{2\pi
i\theta}}\cdot\frac{1-e^{-2\pi i(k+1)\theta}}{1-e^{2\pi
i2\theta}}\cdots\frac{1-e^{-2\pi i(k+m-1)\theta}}{1-e^{2\pi
im\theta}}
\end{equation}
and
\begin{equation}\label{E:2.5}
\alpha_{-k-m,m}=\alpha_{-k,0}\frac{1-e^{-2\pi ik\theta}}{1-e^{2\pi
i\theta}}\cdot\frac{1-e^{-2\pi i(k+1)\theta}}{1-e^{2\pi
i2\theta}}\cdots\frac{1-e^{-2\pi i(k+m-1)\theta}}{1-e^{2\pi
im\theta}}
\end{equation}
for $m\geq 0$ and $k\geq 1$.

Let $k=1$ in equation (\ref{E:2.4}). We have
\begin{equation}\label{E:2.*}
|\alpha_{m,-1-m}|=|\alpha_{0,-1}|\frac{|1-e^{-2\pi
i\theta}|}{|1-e^{2\pi i\theta}|}\cdot\frac{|1-e^{-2\pi
i2\theta}|}{|1-e^{2\pi i2\theta}|}\cdots\frac{|1-e^{-2\pi
ik\theta}|}{|1-e^{2\pi im\theta}|}=|\alpha_{0,-1}|.
\end{equation}
In general for $k\geq 2$ and $m\geq 0$,
\begin{equation}\label{E:key observation}
 |\alpha_{m,-k-m}|=|\alpha_{0,-k}|\frac{|1-e^{-2\pi i(m+1)\theta}|}{|1-e^{2\pi i\theta}|}\cdot \frac{|1-e^{-2\pi i(m+2)\theta}|}{|1-e^{2\pi i2\theta}|}\cdots \frac{|1-e^{-2\pi i(m+k-1)\theta}|}{|1-e^{2\pi i (k-1)\theta}|}.
\end{equation}
To prove $\alpha_{0,-k}=0$ and therefore $\alpha_{m,-k-m}=0$ (by
equation (\ref{E:2.4})), we use the following fact:
\begin{equation}\label{E:goes to zero}
 \sum_{m,n}|\alpha_{m,n}|^2<+\infty\Rightarrow \sum_{m>0}|\alpha_{m,-k-m}|^2<+\infty\Rightarrow \lim_{m\rightarrow+\infty}|\alpha_{m,-k-m}|=0.
\end{equation}
 If $k=1$, then $|\alpha_{m,-1-m}|=|\alpha_{0,-1}|$ by (\ref{E:2.*}).  By (\ref{E:goes to zero}), we have $|\alpha_{m,-1-m}|=|\alpha_{0,-1}|=0$ for all $m\geq 0$.

To prove the general case, we need to use a property of irrational
rotation. Namely, there exists a sequence of increasing integers
$m_n$ such that
\[
 \lim_{n\rightarrow+\infty} e^{2\pi i m_n\theta}=1.
\]
Now for each fixed $k\geq 2$, by (\ref{E:goes to zero}) and
equation~(\ref{E:key observation}),
\begin{eqnarray}
 \nonumber 0&=&\lim_{n\rightarrow+\infty}|\alpha_{m_n,-k-m_n}|\\
\nonumber&=& \lim_{n\rightarrow+\infty} |\alpha_{0,-k}|\frac{|1-e^{-2\pi i(m_n+1)\theta}|}{|1-e^{2\pi i\theta}|}\cdot \frac{|1-e^{-2\pi i(m_n+2)\theta}|}{|1-e^{2\pi i2\theta}|}\cdots \frac{|1-e^{-2\pi i(m_n+k-1)\theta}|}{|1-e^{2\pi i (k-1)\theta}|}\\
\nonumber&=&\lim_{n\rightarrow+\infty} |\alpha_{0,-k}|\frac{|e^{2\pi im_n\theta}-e^{-2\pi i\theta}|}{|1-e^{2\pi i\theta}|}\cdot \frac{|e^{2\pi im_n\theta}-e^{-2\pi i2\theta}|}{|1-e^{2\pi i2\theta}|}\cdots \frac{|e^{2\pi im_n\theta}-e^{-2\pi i(k-1)\theta}|}{|1-e^{2\pi i (k-1)\theta}|}\\
\nonumber&=&|\alpha_{0,-k}|\frac{|1-e^{-2\pi i\theta}|}{|1-e^{2\pi i\theta}|}\cdot \frac{|1-e^{-2\pi i2\theta}|}{|1-e^{2\pi i2\theta}|}\cdots \frac{|1-e^{-2\pi i(k-1)\theta}|}{|1-e^{2\pi i (k-1)\theta}|}\\
\nonumber&=&|\alpha_{0,-k}|.
\end{eqnarray}
By equation~(\ref{E:key observation}),
$|\alpha_{m,-k-m}|=|\alpha_{0,-k}|=0$ for all $m\geq 0$ and $k\geq
1$. By equation (\ref{E:2.5}) and  similar arguments,
$|\alpha_{-k-m,m}|=|\alpha_{-k,0}|=0$ for all $m\geq 0$ and $k\geq
1$.

Above all, we have proved that $\alpha_{m,n}=0$ if either $m<0$ or
$n<0$. Hence
\[
x=\sum_{m\geq 0,n\geq 0}\alpha_{m,n}u^mv^n.
\]
For $k\geq 0$, let $x_k=\sum_{m\geq 0,n\geq 0,
m+n=k}\alpha_{m,n}u^mv^n$. Then $x=\sum_{k=0}^\infty x_k$ as a
vector in $L^2(R)$. Since $x\in \{u+v\}'\cap R$ and
$\{u^mv^n:\,m,n\in\mathbb{Z}\}$ is an orthonormal basis of $L^2(R)$, $x_k\in
\{u+v\}'\cap R$. By equation (\ref{E:2.3}), $\alpha_{m,k-m}$ is
uniquely determined by $\alpha_{0,k}$ for $0\leq k\leq m$. Since
$(u+v)^k$ commutes with $u+v$, $x_k=\lambda_k (u+v)^k$ for some
complex number $\lambda_k$. This implies that $x=\sum_{k=0}^\infty
\lambda_k (u+v)^k$ and the decomposition is unique.

 Suppose $x\in \{u+v\}'\cap R$. Let $x^2=\sum_{k=0}^\infty \sigma_k(u+v)^k$. For $a,b\in R$, let $\langle a,b\rangle=\tau(b^*a)$. Then
 \begin{equation}\label{E:x^2}
\sigma_k=\langle x^2, u^k\rangle=\langle
x,x^*u^k\rangle=\left\langle \sum_{j=0}^\infty \lambda_j(u+v)^j,
\sum_{j=0}^\infty
\bar{\lambda_j}\left((u+v)^*\right)^ju^k\right\rangle=\sum_{j=0}^k\lambda_j\lambda_{k-j},\quad
\forall k\geq 0.
 \end{equation}
If $x^2=x$, then $\lambda_k=\sigma_k$ for all $k$. Let $k=0$. Then
(\ref{E:x^2}) implies that $\lambda_0=\lambda_0^2$. So
$\lambda_0=0$ or $\lambda_0=1$. By considering $1-x$, we may
assume that $\lambda_0=0$. Let $k=1$. Then (\ref{E:x^2}) implies
that
$\lambda_1=\lambda_0\cdot\lambda_1+\lambda_1\cdot\lambda_0=0$. By
(\ref{E:x^2}) and induction, we have $\lambda_k=0$ for all $k\geq
0$. This implies that $x=0$, which completes the proof.
\end{proof}

By the Riesz spectral decomposition theorem, we immediately have
the following corollary.
\begin{Corollary}\label{C:connected}
For every irrational number $\theta\in (0,1)$, the spectrum of
$u+v$ is connected.
\end{Corollary}

\begin{Remark}
\emph{By the proof of Theorem~\ref{L:strongly irreducible}, every
operator in the commutant algebra of $u+v$ can be written as a
formal series $\sum_{n=0}^\infty a_n(u+v)^n$. A similar argument
can show that for $0<\lambda<1$, every operator in the commutant
algebra of $u+\lambda v$ can be written as a formal series
$\sum_{n=-\infty}^\infty a_n(u+\lambda v)^n$.}
\end{Remark}

In the following, we will construct more examples of strongly
irreducible operators relative to the hyperfinite type ${\rm
II}_1$ factor. Precisely, we will prove the following result.

\begin{Proposition}\label{T:u+v^k}
For $\theta$ in a second category subset of $[0,1]$, we have
$u+v^k$ is strongly irreducible relative to $R$ for all
$k=1,2,\cdots$.
\end{Proposition}
To prove Proposition \ref{T:u+v^k}, we need the following lemma.
\begin{Lemma}\label{L:u+v^k}
Let
\[
f_{s,r,k}(z)=\prod_{t=1}^s\frac{1-z^{kt+r}}{1-z^{kt}},\quad
E_{r,k}=\{z\in \mathbb{T}:\, \lim_{s\rightarrow
+\infty}f_{s,r,k}(z)=0\},
\]
where $k$ and $r$ are positive integers such that $k\geq 2$ and
$r\neq 0\, {\rm mod}\, k$. Then $E_{r,k}$ is a first category
subset of $\mathbb{T}$.
\end{Lemma}
\begin{proof}
Let $\epsilon>0$. Note that $f_{s,r,k}(z)$ is a meromorphic
function with finite poles on $\mathbb{T}$. So the set
\[
D_{s,r,k,\epsilon}\triangleq\{z\in \mathbb{T}:\,
|f_{s,r,k}(z)|\leq \epsilon\}
\]
is a closed subset of $\mathbb{T}$. Let
\[
E_{s,r,k,\epsilon}\triangleq\{z\in \mathbb{T}:\,
|f_{s,r,k}(z)|\leq \epsilon,\,\forall a\geq s\}.
\]
Then $\displaystyle E_{s,r,k,\epsilon}=\bigcap_{a\geq
s}D_{a,r,k,\epsilon}$ is also a closed subset of $\mathbb{T}$.

Let $F_{s,r,k,\epsilon}=\mathbb{T}\setminus E_{s,r,k,\epsilon}$.
Then $F_{s,r,k,\epsilon}$ is an open subset of $\mathbb{T}$, and
\begin{eqnarray}
\nonumber F_{s,r,k,\epsilon}&=&\mathbb{T}\setminus \cap_{a\geq s}D_{a,r,k,\epsilon}\\
\nonumber &=&\bigcup_{a\geq s}(\mathbb{T}\setminus D_{a,r,k,\epsilon})\\
\nonumber &=&\bigcup_{a\geq s}\{z\in \mathbb{T}:\, |f_{a,r,k}(z)|>\epsilon\}\\
\nonumber &\supseteq& \bigcup_{a\geq s}\,\{\text{poles of $f_{a,r,k}(z)$}\}\\
\nonumber &=&\bigcup_{a\geq s}\{z:\, z^{ak}=1\}.
\end{eqnarray}
So $F_{s,r,k,\epsilon}$ is a dense open subset of $\mathbb{T}$,
which implies that $E_{s,r,k,\epsilon}$ is a nowhere dense closed
subset of $\mathbb{T}$. Therefore, $\displaystyle E_{r,k}\subseteq
\bigcup_{s=1}^\infty E_{s,r,k,\epsilon}$ is a first category
subset of $\mathbb{T}$.
\end{proof}

\noindent\emph{Proof of Proposition}~\ref{T:u+v^k}:\,
 Define $f_{s,r,k}(z)$ and $E_{r,k}$ as in Lemma~\ref{L:u+v^k}. By Lemma~\ref{L:u+v^k}, $E_{r,k}$ is a first category subset of $\mathbb{T}$. So $\displaystyle \bigcup_{r\neq 0({\rm mod} k)} E_{r,k}$ is also a first category subset of $\mathbb{T}$. Hence
\[
\mathbb{T}\setminus \{e^{2\pi i\theta}:\,\theta\in [0,1]
\,\text{is an rational number}\}\setminus \bigcup_{r\neq 0({\rm
mod} k)} E_{r,k}
\]
is a second category subset of $\mathbb{T}$. Choose a $\theta\in
[0,1]$ such that $e^{2\pi i\theta}$ is in the above set. Then for
all $r$ with $r\neq 0\, {\rm mod}\, k$,
\[
\lim_{s\rightarrow +\infty}f_{s,r,k}(z)=0
\]
does not hold.

 Let  $x\in \{u+v^k\}'\cap R$ be an idempotent. By condition 3 above Theorem~\ref{L:strongly irreducible}, $x=\sum_{m,n\in \mathbb{Z}}\alpha_{m,n}u^mv^n$ and $\sum_{m,n\in \mathbb{Z}}|\alpha_{m,n}|^2=\tau(x^*x)<\infty$. By condition 2 above Theorem~\ref{L:strongly irreducible},
\begin{equation}\label{E:(u+v^k)x}
(u+v^k)x=(u+v^k)\sum_{m,n\in
\mathbb{Z}}\alpha_{m,n}u^mv^n=\sum_{m,n\in
\mathbb{Z}}\alpha_{m,n}u^{m+1}v^n+\sum_{m,n\in
\mathbb{Z}}\alpha_{m,n}e^{2\pi i{km}\theta}u^mv^{n+k}
\end{equation}
and
\begin{equation}\label{E:x(u+v^k)}
x(u+v^k)=\sum_{m,n\in
\mathbb{Z}}\alpha_{m,n}u^mv^n(u+v)=\sum_{m,n\in
\mathbb{Z}}\alpha_{m,n}e^{2\pi in\theta}u^{m+1}v^n+\sum_{m,n\in
\mathbb{Z}}\alpha_{m,n}u^mv^{n+k}.
\end{equation}
 By condition 3 above Theorem~\ref{L:strongly irreducible}, $\{u^mv^n:\,m,n\in\mathbb{Z}\}$ is an orthonormal of $L^2(R)$. Comparing the coefficients of the term $u^mv^n$ in (\ref{E:(u+v^k)x}) and (\ref{E:x(u+v^k)}), we have
\[
\alpha_{m-1,n}+\alpha_{m,n-k}e^{2\pi
i(km)\theta}=\alpha_{m-1,n}e^{2\pi in\theta}+\alpha_{m,n-k},
\]
which is equivalent to
\begin{equation}\label{E:2.10}
\alpha_{m-1,n}(1-e^{2\pi in\theta})=\alpha_{m,n-k}(1-e^{2\pi
i(km)\theta}).
\end{equation}

Since $\theta$ is an irrational number, $1-e^{2\pi ik\theta}\neq
0$ for $k\neq 0$. Let $n=0$ in equation~(\ref{E:2.10}). We have
$\alpha_{m,-k}=0$ for $m\neq 0$. Let $n=-k$ in equation
(\ref{E:2.10}). We have $\alpha_{m,-2k}=0$ for $m\neq 0$, $m\neq
1$. In general, let $n=-sk$ in equation (\ref{E:2.10}). We have
$\alpha_{m,-(s+1)k}=0$ for $m\neq 0$, $\cdots$, $m\neq s$. On the
other hand, let $m=0$ in equation (\ref{E:2.10}). We have
$\alpha_{-1,n}=0$ for $n\neq 0$. Similarly, in general we have
$\alpha_{-s-1,n}=0$ for $n\neq 0, k,\cdots,sk$.

Claim that $\alpha_{m,n}=0$ if either $m<0$ or $n<0$. By the above
arguments, we only need to show that $\alpha_{0,-r}=0$ and
$\alpha_{-r,0}=0$ for $r\geq 1$.  Firstly, we show that
$\alpha_{-r,0}=0$ for $r\geq 1$.

In equation (\ref{E:2.10}), let $m=-r$ and $n=k$. We have
\begin{equation}\label{E:2.11}
\alpha_{-r-1,k}(1-e^{2\pi i(k)\theta})=\alpha_{-r,0}(1-e^{-2\pi
i(kr)\theta}).
\end{equation}
In equation (\ref{E:2.10}), let $m=-r-1$ and $n=2k$. We have
\begin{equation}\label{E:2.12}
\alpha_{-r-2,2k}(1-e^{2\pi
i(2k)\theta})=\alpha_{-r-1,k}(1-e^{-2\pi i(k(r+1))\theta}).
\end{equation}
In general, for a positive integer $s$, let $m=-r-s+1$ and
$n=(s-1)k$ in equation (\ref{E:2.10}). We have
\begin{equation}\label{E:2.13}
\alpha_{-r-s,sk}(1-e^{2\pi
i(sk)\theta})=\alpha_{-r-s+1,k}(1-e^{-2\pi i(k(r+s-1))\theta}).
\end{equation}

By equations (\ref{E:2.11}), (\ref{E:2.12}), (\ref{E:2.13}),
\[
\alpha_{-r-s,sk}=\alpha_{-r,0}\prod_{t=1}^s\frac{1-e^{-2\pi
i(k(r+s-t))\theta}}{1-e^{2\pi i(tk)\theta}}.
\]
So for $s>r-1$, we have
\begin{equation}\label{E:2.14}
|\alpha_{-r-s,sk}|=|\alpha_{-r,0}|\prod_{t=1}^{r-1}\frac{|1-e^{-2\pi
i((r+s))k\theta}e^{2\pi i(tk))\theta}|}{|1-e^{2\pi i(tk)\theta}|}.
\end{equation}
Since $\theta$ is an irrational number, there is a sequence
positive integers $s_n$ such that
\[
 \lim_{n\rightarrow\infty} e^{-2\pi i (r+s_n)k\theta}=1.
\]

By equation (\ref{E:2.14}),
$\lim_{n\rightarrow\infty}|\alpha_{-r-s_n,s_nk}|=|\alpha_{-r,0}|$.
Since $\sum_{n=1}^\infty|\alpha_{-r-s_n,s_nk}|^2<\infty$,
$|\alpha_{-r,0}|=\lim_{n\rightarrow\infty}|\alpha_{-r-s_n,s_nk}|=0$.

Secondly, we show that $\alpha_{0,-r}=0$ for $r\geq 1$. In
equation (\ref{E:2.10}), let $m=1$ and $n=-r$. We have
\begin{equation}\label{E:2.15}
\alpha_{0,-r}(1-e^{2\pi i(-r)\theta})=\alpha_{1,-r-k}(1-e^{2\pi
i(k)\theta}).
\end{equation}
In equation (\ref{E:2.10}), let $m=2$ and $n=-r-k$. We have
\begin{equation}\label{E:2.16}
\alpha_{1,-r-k}(1-e^{2\pi
i(-r-k)\theta})=\alpha_{2,-r-2k}(1-e^{2\pi i(2k)\theta}).
\end{equation}
In general, for a positive integer $s$, let $m=s+1$ and
$n=-r-(s-1)k$ in equation (\ref{E:2.10}). We have
\begin{equation}\label{E:2.17}
\alpha_{0,-r-(s-1)k}(1-e^{2\pi
i(-r-(s-1)k)\theta})=\alpha_{s,-r-sk}(1-e^{2\pi i(sk)\theta}).
\end{equation}

By equations (\ref{E:2.15}), (\ref{E:2.16}), (\ref{E:2.17}),
\begin{equation}\label{E:2.18*}
\alpha_{s,-r-sk}=\alpha_{0,-r}\prod_{t=1}^s\frac{1-e^{-2\pi
i(tk+(r-k))\theta}}{1-e^{2\pi i(tk)\theta}}.
\end{equation}
We consider two cases. Case 1: $r=0 ({\rm mod}\, k)$. By equation
(\ref{E:2.18*}),
\begin{equation}\label{E:2.18}
\alpha_{s,-r-sk}=\alpha_{0,-r}\prod_{t=1}^{\frac{r-k}{k}}\frac{1-e^{-2\pi
i(sk+r)\theta}e^{2\pi i(tk)\theta}}{1-e^{2\pi i(tk)\theta}}
\end{equation}
 for $s>\frac{r-k}{k}$. Since $\theta$ is an irrational number, there is a sequence positive integers $s_n$ such that
\[
 \lim_{n\rightarrow\infty} e^{-2\pi i (s_nk+r)\theta}=1.
\]
By equation (\ref{E:2.18}),
$\lim_{n\rightarrow\infty}|\alpha_{s_n,-r-s_nk}|=|\alpha_{0,-r}|$.
Since $\sum_{n=1}^\infty|\alpha_{s_n,-r-s_nk}|^2<\infty$,
\[|\alpha_{0,-r}|=\lim_{n\rightarrow\infty}|\alpha_{s_n,-r-s_nk}|=0.\]

Case 2: $r\neq 0 ({\rm mod}\, k)$. Note that $\sum_{s=1}^\infty
|\alpha_{s,-r-sk}|^2<\infty$. So $\lim_{s\rightarrow\infty}
\alpha_{s,-r-sk}=0$. By the choice of $\theta$,
\[\lim_{s\rightarrow\infty}\prod_{t=1}^s\frac{1-e^{-2\pi i(tk+(r-k))\theta}}{1-e^{2\pi i(tk)\theta}}=0\]
does not hold. So $\alpha_{0,-r}$ has to be 0.

Above all, we have proved that $\alpha_{m,n}=0$ if either $m<0$ or
$n<0$. Furthermore, we claim that $\alpha_{m,n}=0$ for $m,n\geq 0$
and $n\neq 0 ({\rm mod}\, k)$. Let $s$ be the least positive
integer greater than $n/k$. By equation (\ref{E:2.10}), we have
\begin{eqnarray}
\nonumber \alpha_{m,n}(1-e^{2\pi i n\theta})&=&\alpha_{m+1, n-k}(1-e^{2\pi ik(m+1)\theta}),\\
\nonumber \alpha_{m+1,n-k}(1-e^{2\pi i (n-k)\theta})&=&\alpha_{m+2, n-2k}(1-e^{2\pi ik(m+2)\theta}),\\
\nonumber &\vdots&\\
\nonumber \alpha_{m+s-1,n-(s-1)k}(1-e^{2\pi i
(n-(s-1)k)\theta})&=&\alpha_{m+s, n-sk}(1-e^{2\pi ik(m+s)\theta}).
\end{eqnarray}
Since $n-sk<0$, $\alpha_{m+s, n-sk}=0$. The above equations imply
that $\alpha_{m,n}=0$ since $1-e^{2\pi i (n-jk)\theta}\neq 0$ for
all $j$.

Hence
\[
x=\sum_{m\geq 0,n\geq 0}\alpha_{m,n}u^mv^{kn},
\]  which implies that $x$ is in the commutant algebra of $u+v^k$ relative to the von Neumann subalgebra generated by $u$ and $v^k$. Since $v^ku=e^{2\pi ik\theta}v^ku$ and $k\theta$ is an irrational number,
 $x=0$ or $x=1$ by Theorem~\ref{L:strongly irreducible}. So $T$ is a strongly irreducible operator relative to $R$. This completes the proof of Theorem~\ref{T:u+v^k}.
\vskip 0.5cm
\begin{Proposition}\label{indexn}
Let $n$ be a positive integer. Then by Theorem~\ref{T:generators} $N=W^*(u+v^n)=W^*(u,v^n)$ is
an irreducible subfactor of $W^*(u+v)=R$ with Jones
index~\cite{Jo} $[R:N]=n$.
\end{Proposition}
\begin{proof}
 Since $R$ is generated by $u,v$ and $N$ is generated by $u,v^n,$ it is clear that every element of $R$ can be written as finite linear combinations of elements in $Nv^i, 0\leq i\leq n-1.$
 Since $Nv^i$ is orthogonal to $Nv^j, 0\leq i\neq j\leq n-1,$ under the inner product defined by the trace on $R,$ it follows that
 $R=N\oplus
Nv\oplus Nv^2\oplus \cdots\oplus Nv^{n-1}$, where $Nv^i$ is orthogonal to $Nv^j, 0\leq i\neq j\leq n-1.$ So
by~\cite{PP}, $v^i, 0\leq i\leq n-1$ is a Pimsner-Popa basis of $R$ over $N$, and since $v^i$ is unitary, it follows that
$[R:N]=n$.
\end{proof}
On the other hand, by Proposition~\ref{T:u+v^k} for $\theta$ in a
second category subset of $[0,1]$, $u+v^n$ is strongly irreducible
relative to $R$. So for every bounded invertible operator $x\in
R$, $x(u+v^n)x^{-1}$ generates an irreducible subfactor
$W^*(x(u+v^n)x^{-1})$ of $R$. Is it true that
  $[R:W^*(x(u+v^n)x^{-1})]=n$ for all bounded invertible
 operators $x$ in $R$, at least when $x$ is close to identity in norm?

\vskip 0.5cm

By definitions if $T$ is
strongly irreducible relative to $M$, then $T$ is irreducible
relative to $M$. An operator $T$ is strongly irreducible relative to a type ${\rm
II}_1$ factor  if and only if  $XTX^{-1}$ is an irreducible
operator relative to $M$ for every bounded invertible operator
$X\in M$. However, if $T$ is irreducible relative to $M$, this is
not true in general. The following result shows that an
irreducible operator relative to $M$ can be similar to a unitary
operator.

\begin{Proposition}
Let $\theta$ be an irrational number in $[0,1]$ and let $n$ be any
positive integer. Then in the hyperfinite type ${\rm II}_1$ factor
$R$ there exists a bounded invertible operator $x$ such that
$W^*(xux^{-1})=W^*(u+v^n)=W^*(u,v^n)$.
\end{Proposition}
\begin{proof}
Let $\sigma$ be a nonempty open connected subset of
$\sigma(v^n)=\mathbb{T}$ such that $\sigma\cap e^{2\pi
in\theta}\sigma=\emptyset$. Let $x=1-\frac{E_{v^n}(\sigma)}{2}\in
R$, where $E_{v^n}(\cdot)$ denotes the spectral measure of $v^n$.
Since $v^nu=e^{2\pi in\theta}uv^n$, $f(v^n)u=uf(e^{2\pi
in\theta}v)$ for all $f\in L^\infty(\mathbb{T},m)$. Therefore,
\begin{equation}
E_{v^n}(\sigma)u=uE_{v^n}(e^{2\pi in\theta}\sigma),
\end{equation}
which implies that
\begin{equation}
u^{-1}E_{v^n}(\sigma)=E_{v^n}(e^{2\pi in\theta}\sigma)u^{-1}.
\end{equation}
Similarly, by $v^nu^{-1}=e^{-2\pi in\theta}u^{-1}v^n$, we have
\begin{equation}
uE_{v^n}(\sigma)=E_{v^n}(e^{-2\pi in\theta}\sigma)u.
\end{equation}

Combining the above equations, we have
\[
(xux^{-1})^*(xux^{-1})=x^{-1}u^{-1}xxux^{-1}=\left(1-\frac{E_{v^n}(\sigma)}{2}\right)^{-1}
\left(1-\frac{E_{v^n}(e^{2\pi in\theta}\sigma)}{2}\right)^2
\left(1-\frac{E_{v^n}(\sigma)}{2}\right)^{-1}.
\]
We can write
\[
1-\frac{E_{v^n}(e^{2\pi in\theta}\sigma)}{2}=\left(
\begin{array}{ccc}
  1 & 0 & 0 \\
  0 & \frac{1}{2} & 0 \\
  0 & 0 & 1 \\
\end{array}
\right)\begin{array}{l}
\text{Ran}\,E_{v^n}(\mathbb{T}\setminus (e^{2\pi in\theta}\sigma\cup \sigma))\\
\text{Ran}\,E_{v^n}(e^{2\pi in\theta}\sigma)\\
\text{Ran}\,E_{v^n}(\sigma)
\end{array}
\]
and write
\[
1-\frac{E_{v^n}(\sigma)}{2}=\left(
\begin{array}{ccc}
  1 & 0 & 0 \\
  0 & 1 & 0 \\
  0 & 0 & 2 \\
\end{array}
\right)\begin{array}{l}
\text{Ran}\,E_{v^n}(\mathbb{T}\setminus (e^{2\pi in\theta}\sigma\cup \sigma))\\
\text{Ran}\,E_{v^n}(e^{2\pi in\theta}\sigma)\\
\text{Ran}\,E_{v^n}(\sigma)
\end{array}.
\]
So we have
\[
(xux^{-1})^*(xux^{-1})=\left(
\begin{array}{ccc}
  1 & 0 & 0 \\
  0 & \frac{1}{4} & 0 \\
  0 & 0 & 4 \\
\end{array}
\right)\begin{array}{l}
\text{Ran}\,E_{v^n}(\mathbb{T}\setminus (e^{2\pi in\theta}\sigma\cup \sigma))\\
\text{Ran}\,E_{v^n}(e^{2\pi in\theta}\sigma)\\
\text{Ran}\,E_{v^n}(\sigma)
\end{array}\in W^*(xux^{-1}).
\]
Therefore,
\[
\left(
\begin{array}{ccc}
  1 & 0 & 0 \\
  0 & \frac{1}{2} & 0 \\
  0 & 0 & 2 \\
\end{array}
\right)\begin{array}{l}
\text{Ran}\,E_{v^n}(\mathbb{T}\setminus (e^{2\pi in\theta}\sigma\cup \sigma))\\
\text{Ran}\,E_{v^n}(e^{2\pi in\theta}\sigma)\\
\text{Ran}\,E_{v^n}(\sigma)
\end{array}\in W^*(xux^{-1}).
\]
This implies that
\[
\left(
\begin{array}{ccc}
  0 & 0 & 0 \\
  0 & -\frac{1}{2} & 0 \\
  0 & 0 & 1 \\
\end{array}
\right)\begin{array}{l}
\text{Ran}\,E_{v^n}(\mathbb{T}\setminus (e^{2\pi in\theta}\sigma\cup \sigma))\\
\text{Ran}\,E_{v^n}(e^{2\pi in\theta}\sigma)\\
\text{Ran}\,E_{v^n}(\sigma)
\end{array}\in W^*(xux^{-1}).
\]
Note that
\[
\left(
\begin{array}{ccc}
  0 & 0 & 0 \\
  0 & -\frac{1}{2} & 0 \\
  0 & 0 & 1 \\
\end{array}
\right)^k\overset{\text{SOT}}{\longrightarrow} \left(
\begin{array}{ccc}
  0 & 0 & 0 \\
  0 & 0 & 0 \\
  0 & 0 & 1 \\
\end{array}
\right)\begin{array}{l}
\text{Ran}\,E_{v^n}(\mathbb{T}\setminus (e^{2\pi in\theta}\sigma\cup \sigma))\\
\text{Ran}\,E_{v^n}(e^{2\pi in\theta}\sigma)\\
\text{Ran}\,E_{v^n}(\sigma)
\end{array}=E_{v^n}(\sigma).
\]
Hence, $E_{v^n}(\sigma)\in W^*(xux^{-1})$ and
$x=1-\frac{E_{v^n}(\sigma)}{2}\in W^*(xux^{-1})$. Therefore, $u\in
W^*(xux^{-1})$. Note that
\[
E_{v^n}(e^{2\pi ikn\theta}\sigma)=u^{-k}E_{v^n}(\sigma)u^k\in
W^*(xux^{-1}), \,\forall k\in \mathbb{N}.
\]
Since $\{e^{2\pi ikn\theta}\}_{k=1}^\infty$ is dense in
$\mathbb{T}$, $E_{v^n}(\sigma_1)\in W^*(xux^{-1})$ for every open
connected subset $\sigma_1$ in $\mathbb{T}$ which has same arc
length as $\sigma$. If $\sigma_0$ is an open connected subset of
$\mathbb{T}$ with arc length smaller than the arc length of
$\sigma$, then there are two open connected subsets
$\sigma_1,\sigma_2$ of $\mathbb{T}$ with arc length same as
$\sigma$ such that $\sigma_1\cap \sigma_2=\sigma_0$. Thus
\[E_{v^n}(\sigma_0)=E_{v^n}(\sigma_1)\cap E_{v^n}(\sigma_2)\in
W^*(xux^{-1}).\] This implies that for every measurable subset $F$
of $\mathbb{T}$, we have $E_{v^n}(F)\in W^*(xux^{-1})$. So $v^n\in
W^*(xux^{-1})$ and we have proved that $W^*(xux^{-1})=W^*(u,v^n)$.
\end{proof}

In general, we have the following observation.
\begin{Proposition}
Let $N\subseteq M$ be an inclusion of type ${\rm II}_1$ factors. Then there exists an operator $S\in N$ which is similar to an irreducible operator
$T$ relative to $M$.
\end{Proposition}
\begin{proof}
We may identify $N=M_3(\mathbb{C})\otimes N_1$ and $M=M_3(\mathbb{C})\otimes M_1$. Choose complex numbers $\alpha_1,\alpha_2,\alpha_3$ such that $\alpha_i\neq \alpha_j$ for $i\neq j$.  Let $D$ be an irreducible operator in $M_1$,
\[
S=\begin{pmatrix}
\alpha_1&0&0\\
0&\alpha_2&0\\
0&0&\alpha_3
\end{pmatrix},\quad T=\begin{pmatrix}
\alpha_1&1&D\\
0&\alpha_2&1\\
0&0&\alpha_3
\end{pmatrix},\quad X=\begin{pmatrix}
1&\frac{1}{\alpha_1-\alpha_2}&\frac{1}{(\alpha_1-\alpha_2)(\alpha_1-\alpha_3)}+\frac{D}{\alpha_1-\alpha_3}\\
0&1&\frac{1}{\alpha_2-\alpha_3}\\
0&0&1
\end{pmatrix}.
\]
Then direct calculations show that $T$ is an irreducible operator in $M$ and $XSX^{-1}=T$.
\end{proof}

\section{Spectrum of $u+\lambda v$}

\begin{Theorem}\label{T:1}
For every irrational number $\theta\in (0,1)$,
\[
\sigma(u+\lambda v)=
\begin{cases}
\mathbb{T}& 0<\lambda<1,\\
\overline{B(0,1)}&\lambda=1,\\
\lambda \mathbb{T}& \lambda>1,
\end{cases}
\]
where $\mathbb{T}$ is the unit circle.
\end{Theorem}
\begin{proof}
Note that $u+v=u(1+u^*v)$. Since $u^*v$ is a Haar unitary operator, $-1\in\sigma(u^*v)$. This implies that $u+v$ is not invertible and therefore $0\in\sigma(u+v)$. For every $\theta\in [0,2\pi]$, $e^{i\theta}u$ and $e^{i\theta}v$ satisfy the same irrational rotation relation as $u$ and $v$, so $\sigma(u+v)$ is rotation symmetric with respect to 0. By Corollary~\ref{C:connected}, $\sigma(u+v)$ is a closed disk with center 0. By Lemma~\ref{L:r=1}, $\sigma(u+v)=\overline{B(0,1)}$.

For $0<\lambda<1$, $u+\lambda v=u(1+\lambda u^*v)$ is invertible.
In the following  we  prove that $r\left((u+\lambda v)^{-1}\right)\leq 1$. The proof is similar to the proof of Lemma~\ref{L:r=1}. However, some details should be treated carefully, so we include the complete proof.
By equation (\ref{E:(u+v)^{-1}}),
\[
(u+\lambda v)^{-n}=(1+\lambda w)^{-1}(1+\alpha \lambda w)^{-1}\cdots (1+\alpha^{n-1}\lambda w)^{-1}u^{-n},\quad \forall n\in \mathbb{N}.
\]
Hence,
\begin{eqnarray*}
 \|(u+\lambda v)^{-n}\|^{1/n}&=&\|(1+\lambda w)^{-1}(1+\alpha\lambda w)^{-1}\cdots(1+\alpha^{(n-1)}\lambda w)^{-1}\|^{1/n}\\
&=&\left(\max_{z\in \mathbb{T}}\left|(1+\lambda z)^{-1}(1+\alpha \lambda z)^{-1}\cdots (1+\alpha^{(n-1)}\lambda z)^{-1}\right|\right)^{1/n}\\
&=&\max_{x\in[0,1]}\left(\prod_{k=0}^{n-1}\left(1+\lambda^2+2\lambda\cos(2\pi(x+k\theta))\right)^{-1}\right)^{\frac{1}{2n}}.
\end{eqnarray*}

Let $\epsilon>0$. Note that
\[
\lim_{n\rightarrow\infty}\frac{1}{2n}
\left(\sum_{k=1}^{n}-\ln\left(1+\lambda^2+2\lambda\cos\left(\frac{k\pi}{n}\right)\right)+
\sum_{k=n}^{2n-1}-\ln\left(1+\lambda^2+2\lambda\cos\left(\frac{k\pi}{n}\right)\right)\right)
\]
\[
=-\int_0^1\ln(1+\lambda^2+2\lambda \cos 2\pi x)dx=0.
\]
There is $N\in \mathbb{N}$ such that
\[
\frac{1}{2N}
\left(\sum_{k=1}^{N}-\ln\left(1+\lambda^2+2\lambda\cos\left(\frac{k\pi}{N}\right)\right)+
\sum_{k=N}^{2N-1}-\ln\left(1+\lambda^2+2\lambda\cos\left(\frac{k\pi}{N}\right)\right)\right)<\epsilon/2.
\]
Let
\[
L(\lambda)=\max_{1\leq k\leq 2N-1}\left|\ln\left(1+\lambda^2+2\lambda\cos\left(\frac{k\pi}{N}\right)\right)\right|.
\]
Then for $0<\lambda<1$, $L(\lambda)<\infty$ (Note that if
$\lambda=1$, then $L(\lambda)=\infty$). Divide the unit circle
$\mathbb{T}$ into $2N$ equal sections $A_1,\cdots, A_{2N}$. By Lemma~\ref{L:irrational rotation},  there exists $N'$ such that for all
$n\geq N'$ and all $x\in [0,1]$, if $A_k$ contains
$n/(2N)+r_{k}(x)$ points of $e^{2\pi ix}, \alpha e^{2\pi ix},
\cdots, \alpha^{n-1} e^{2\pi ix}$, then $\displaystyle
\frac{\sum_{k=1}^{2N}|r_k(x)|}{n}<\frac{\epsilon}{L(\lambda)}. $
Note that $\cos 2\pi x$ is decreasing for $x\in [0,1/2]$ and
increasing for $x\in [1/2,1]$. Therefore, for all $x\in [0,1]$,

\[
\frac{1}{n}\sum_{k=0}^{n-1}
-\ln \left(1+\lambda^2+2\lambda\cos(2\pi(x+k\theta))\right)\leq \frac{1}{n}\sum_{k=1}^{N}-\left(\frac{n}{2N}+r_{k}(x)\right)\ln\left(1+\lambda^2+2\lambda\cos\left(\frac{k\pi}{N}\right)\right)
\]
\[
+\frac{1}{n}\sum_{k=N}^{2N-1}-\left(\frac{n}{2N}+r_{k+1}(x)\right)\ln\left(1+\lambda^2+2\lambda\cos\left(\frac{k\pi}{N}\right)\right)
\]
\[
=\frac{1}{2N}
\left(\sum_{k=1}^{N}-\ln\left(1+\lambda^2+2\lambda\cos\left(\frac{k\pi}{N}\right)\right)+
\sum_{k=N}^{2N-1}-\ln\left(1+\lambda^2+2\lambda\cos\left(\frac{k\pi}{N}\right)\right)\right)
\]
\[
+\frac{1}{n}\sum_{k=1}^{N}-r_{k}(x)\ln\left(1+\lambda^2+2\lambda\cos\left(\frac{k\pi}{N}\right)\right)
+\frac{1}{n}\sum_{k=N}^{2N-1}-r_{k+1}(x)\ln\left(1+\lambda^2+2\lambda\cos\left(\frac{k\pi}{N}\right)\right)
\]
\[
<\epsilon+\frac{1}{n}\sum_{k=1}^{2N}
|r_k(x)|L(\lambda)<2\epsilon.
\]
This implies that for all $n\geq N'$ and $x\in [0,1]$,
\[
\left(\prod_{k=0}^{n-1}\left(1+\lambda^2+2\lambda\cos(2\pi(x+k\theta))\right)^{-1}\right)^{\frac{1}{2n}}\leq e^{2\epsilon}.
\]
Therefore, $\|(u+\lambda v)^{-n}\|^{1/n}\leq e^{2\epsilon}$ for
all $n\geq N'$. So $r\left((u+\lambda v)^{-1}\right)\leq
e^{2\epsilon}$. Since $\epsilon>0$ is arbitrary,
$r\left((u+\lambda v)^{-1}\right)\leq 1$. By Lemma~\ref{L:r=1},
$r(u+\lambda v)=1$ for $0<\lambda<1$. This implies that
$\sigma(u+\lambda v)\subseteq \mathbb{T}$. Since $\sigma(u+\lambda
v)$ is rotation invariant, $\sigma(u+\lambda v)=\mathbb{T}$.

If $\lambda>1$, then $\sigma(u+\lambda v)=\lambda
\sigma(\lambda^{-1}u+v)=\lambda \mathbb{T}$. This completes the
proof.

\end{proof}

\section{Brown's spectral distribution of $u+\lambda v$}

Let $M$ be a finite von Neumann algebra with a faithful normal tracial state $\tau$. The \emph{Fuglede-Kadison determinant}~\cite{FK}, $\Delta:\, M\rightarrow [0,+\infty[,$ is given by
\[
\Delta(T)={\rm exp}\{\tau(\ln|T|)\},\quad T\in M,
\]
with ${\rm exp}\{-\infty\}:=0$. For an arbitrary element $T$ in $M$ the function $\lambda\rightarrow \ln\Delta(T-\lambda 1)$ is subharmonic on $\mathbb{C}$, and its Laplacian
\[
d\mu_T(\lambda):=\frac{1}{2\pi}\triangledown^2\ln \Delta(T-\lambda 1),
\]
in the distribution sense, defines a probability measure $\mu_T$ on $\mathbb{C}$, called the \emph{Brown's spectral distribution} or \emph{Brown measure} of $T$. From the definition, Brown measure $\mu_T$ only depends on the joint distribution of $T$ and $T^*$, i.e., the (noncommutative) mixed moments of $T$ and $T^*$.

If $T$ is normal, then $\mu_T$ is the  trace $\tau$ composed with the spectral projections of $T$. If $M=M_n(\mathbb{C})$, then $\mu_T$ is the normalized counting measure $\frac{1}{n}\left(\delta_{\lambda_1}+\delta_{\lambda_2}+\cdots+\delta_{\lambda_n}\right)$, where $\lambda_1,\lambda_2,\cdots,\lambda_n$ are the eigenvalues of $T$ repeated according to root multiplicity.

The following theorem is  Theorem 2.2 of~\cite{HS2}.
\begin{Theorem}\label{T:HS}
Let $T\in M$, and for $n\in \mathbb{N}$, let $\mu_n\in \text{Prob}([0,\infty))$ denote the distribution of $(T^n)^*T^n$ w.r.t $\tau$, and let $\nu_n$ denote the push-forward measure of $\mu_n$ under the map $t\rightarrow t^{\frac{1}{n}}$. Moreover, let $\nu$ denote the push-forward measure of $\mu_T$ under the map $z\rightarrow |z|^2$, i.e., $\nu$ is determined by
\[
\nu([0,t^2])=\mu_T(\overline{B(0,t)}),\quad, t>0.
\]
Then $\nu_n\rightarrow \nu$ weakly in $\text{Prob}([0,\infty))$.
\end{Theorem}

\begin{Theorem}
The Brown measure of $u+\lambda v$ is the Haar measure on the unit
circle $\mathbb{T}$ if $0<\lambda\leq 1$ and the Haar measure on
$\lambda \mathbb{T}$ if $\lambda>1$.
\end{Theorem}
\begin{proof}
By Theorem~\ref{T:1}, $\sigma(u+\lambda v)=\mathbb{T}$ if
$0<\lambda<1$ and $\sigma(u+\lambda v)=\lambda \mathbb{T}$ if
$\lambda>1$. Since $\mu_{(u+\lambda v)}$ is rotation invariant and
the support of $\mu_{(u+\lambda v)}$ is contained in
$\sigma(u+\lambda v)$, the Brown measure of $u+\lambda v$ is the
Haar measure on the unit circle $\mathbb{T}$ if $0<\lambda< 1$ and
the Haar measure on $\lambda \mathbb{T}$ if $\lambda>1$.

In the following, we consider the case $\lambda=1$. Let $T=u+v$, and let $\nu$ and $\nu_n$ be the measures defined as in Theorem~\ref{T:HS}. Note that $((T^n)^*T^n)^{\frac{1}{n}}=|(1+w)\cdots (1+\alpha^{n-1}w)|^{\frac{2}{n}}$, where $w=u^*v$ is a Haar unitary operator. So we can view $((T^n)^*T^n)^{\frac{1}{n}}$ as the multiplication operator on $L^2[0,1]$ corresponding to the function
\[
\left|\prod_{k=0}^{n-1}\left(2+2\cos(2\pi(x+k\theta))\right)\right|^{\frac{1}{n}}.
\]
Let $m$ be the Lebesgue measure on $[0,1]$.  For $0<b<1$, since $[0,b)$ is an open set relative to $[0,\infty)$ and   $\nu_n\rightarrow \nu$ weakly in $\text{Prob}([0,\infty))$ (by Theorem~\ref{T:HS}),
\[
\nu([0,b))\leq \liminf_{n\rightarrow\infty}\nu_n([0,b))=\liminf_{n\rightarrow\infty}m\left(\left\{x: \left|\prod_{k=0}^{n-1}\left(2+2\cos(2\pi(x+k\theta))\right)\right|^{\frac{1}{n}}\in [0,b)\right\}\right).
\]
By Lemma~\ref{L:Ergodic},  for almost all $x\in [0,1]$,
\[
\lim_{n\rightarrow\infty}\left|\prod_{k=0}^{n-1}\left(2+2\cos(2\pi(x+k\theta))\right)\right|^{\frac{1}{n}}=1.
\]
In particular, $\left|\prod_{k=0}^{n-1}\left(2+2\cos(2\pi(x+k\theta))\right)\right|^{\frac{1}{n}}$ converges in
 measure to the constant function $1$ on [0,1]. Since $b<1$, $\nu([0,b))=0$. Let $r'(u+v)$ be the
 Brown spectral radius of $u+v$. Then $r'(u+v)\leq r(u+v)=1$ (see~\cite{HS2}, Corollary 2.6). So the support of $\nu$ is contained in $[0,1]$. Thus $\nu$ is the Dirac measure  $\delta_1$ and the support of $\mu_T$ is contained in $\mathbb{T}$.  Since $\mu_T$ is rotation invariant, $\mu_T$ is the Haar measure on $\mathbb{T}$.
\end{proof}

\noindent Junsheng Fang

\noindent School of Mathematical Sciences,
Dalian University of Technology,

\noindent Dalian, 116024, P.R China,

\noindent {\em E-mail address: } [Junsheng Fang]\,\,
junshengfang\@@gmail.com

\vspace{.2in}

\noindent Chunlan Jiang

\noindent Department of Mathematics, Hebei Normal University,

\noindent Shijianzhuang, 050016, P.R China

\noindent {\em E-mail address: } [Chunlan Jiang]\,\,
cljiang\@@mail.hebtu.edu.cn

\newpage

\noindent Huaxin Lin

\noindent Department of Mathematics, East China  Normal University

\noindent Shanghai, 200241, P.R. China, and

\noindent University of Oregon

\noindent Eugene, Oregon, 97402, U.S.A.

\noindent {\em E-mail address: } [Huaxin Lin]\,\,
hlin\@@uoregon.edu

\vspace{.2in}

\noindent Feng Xu

\noindent Department of Mathematics, University of California at Riverside

\noindent Riverside, CA 92521, United States

\noindent {\em E-mail address: } [Feng Xu]\,\,
xufeng\@@math.ucr.edu

\end{document}